\documentclass[11pt, a4paper]{article}

\usepackage[utf8]{inputenc}

\usepackage[dvipsnames]{xcolor}
\usepackage[backref=page]{hyperref}

\usepackage[english]{babel}

\usepackage[margin=2.5cm]{geometry}
\usepackage{mathtools}
\usepackage{amssymb}
\usepackage{amsthm}
\usepackage{amsmath}
\usepackage{empheq}

\usepackage{array,multirow,blkarray}

\usepackage[toc,page]{appendix}

\usepackage[nodayofweek]{datetime}

\usepackage{textcomp}

\usepackage{etoolbox}

\newcommand{\boldf}{\mathbf{f}}
\newcommand{\bd}{\mathbf{d}}
\newcommand{\p}{\mathbf{p}}
\newcommand{\q}{\mathbf{q}}

\usepackage{bbm}
\usepackage[]{dsfont}
\usepackage{setspace}
\usepackage{xfrac}
\usepackage{stmaryrd}
\usepackage[shortlabels]{enumitem}

\usepackage{upgreek}

\newcommand{\V}{\mathcal{V}}

\newcommand{\bp}{\mathbf{p}}
\newcommand{\bq}{\mathbf{q}}

\usepackage[linesnumbered,commentsnumbered,ruled,vlined]{algorithm2e}

\SetKw{run}{run}

\usepackage{bm}

\DeclareMathOperator{\diag}{diag}

\DeclareMathOperator{\tr}{tr}

\usepackage{accents}

\usepackage{cleveref}

\newcommand{\rk}{\operatorname{rk}}

\usepackage[noadjust]{cite}
\usepackage{booktabs}

\usepackage{algorithmic}
\usepackage{caption}
\usepackage{pgfplots}
\pgfplotsset{compat=1.11}

\usepackage[bottom]{footmisc}

\newcommand{\im}{\operatorname{Im}}

\newcommand\myshade{100}
\colorlet{mylinkcolor}{NavyBlue}
\colorlet{mycitecolor}{YellowOrange}
\colorlet{myurlcolor}{Aquamarine}

\hypersetup{
  linkcolor  = mylinkcolor!\myshade!black,
  citecolor  = mycitecolor!\myshade!black,
  urlcolor   = myurlcolor!\myshade!black,
  colorlinks = true,
}

\usepackage[bibliography=common]{apxproof}

\newtheorem{theorem}{Theorem}[section]
\newtheorem{definition}[theorem]{Definition}
\newtheorem{lemma}[theorem]{Lemma}

\newtheorem{fact}[theorem]{Fact}

\newtheoremrep{apxtheorem}[theorem]{Theorem}
\newtheoremrep{apxlemma}[theorem]{Lemma}
\newtheoremrep{apxfact}[theorem]{Fact}

\newtheoremstyle{case}{}{}{\itshape}{1em}{}{:}{ }{}
\theoremstyle{case}

\definecolor{darkred}{RGB}{180, 0, 0}
\definecolor{darkgreen}{rgb}{0, 0.6, 0}
\definecolor{darkblue}{RGB}{51,51,178}
\definecolor{lightgray}{RGB}{231,231,231}
\definecolor{lightblue}{RGB}{180,180,254}
\definecolor{lightred}{HTML}{FEB4B4}
\definecolor{darkcyan}{HTML}{7FBFBF}

\definecolor{b2}{RGB}{51,153,255}
\definecolor{mygreen}{RGB}{80,180,0}
\definecolor{yl}{RGB}{255,80,0}

\newcommand{\1}{1}

\newcommand{\supp}{\mathrm{supp}}
\newcommand{\Tmat}{{\cal T}_{\mathrm{mat}}}

\DeclareMathOperator{\new}{new}
\DeclareMathOperator{\midd}{mid}

\newcommand{\vnd}{\V_{n,d}}
\newcommand{\vntd}{\V_{n,2d}}

\renewcommand{\phi}{\varphi}
\renewcommand{\rho}{\varrho}

\newcommand{\R}{\mathbb{R}}

\newcommand{\wh}{\widehat}
\newcommand{\wt}{\widetilde}
\newcommand{\ov}{\overline}

\newcommand{\pr}[2]{\langle #1, #2 \rangle}

\newcommand{\OPT}{\mathrm{OPT}}

\newcounter{Hequation}

\makeatletter
\g@addto@macro\equation{\stepcounter{Hequation}}
\makeatother

\title{
A Faster Interior-Point Method for Sum-of-Squares Optimization} 

\author{Shunhua Jiang \thanks{Supported by NSF CAREER award CCF-1844887.} \\ 
Columbia University\\ \texttt{sj3005@columbia.edu}
\and Bento Natura\thanks{This project has received funding from the European Research Council (ERC) under the European Union’s Horizon 2020 research and innovation programme (grant agreement no. 757481–ScaleOpt).} \\ 
London School of Economics\\
\texttt{b.natura@lse.ac.uk} 
\and Omri Weinstein\thanks{Supported by NSF CAREER award CCF-1844887 and ISF grant \#3011005535.}\\ The Hebrew University and Columbia University\\ \texttt{omri@cs.columbia.edu}
}

\newcommand{\K}{\mathcal K}
\newcommand{\cS}{\mathcal S}

\date{}



\begin{document}
\maketitle

\begin{abstract}

We present a faster interior-point method 
for optimizing sum-of-squares (SOS) polynomials, which are a central tool in polynomial optimization and capture convex programming in the Lasserre hierarchy. 
Let $p = \sum_i q^2_i$ be an $n$-variate SOS polynomial of degree $2d$. Denoting by $L := \binom{n+d}{d}$ and $U := \binom{n+2d}{2d}$ the dimensions of the vector spaces in which $q_i$'s and $p$ live respectively, our algorithm runs in time $\tilde{O}(LU^{1.87})$. This is polynomially faster than state-of-art SOS and semidefinite programming solvers, which achieve runtime $\tilde{O}(L^{0.5}\min\{U^{2.37}, L^{4.24}\})$. 

The centerpiece of our algorithm is a dynamic data structure for maintaining the inverse of the Hessian of the SOS barrier function under the \emph{polynomial interpolant basis}, which efficiently extends to multivariate SOS optimization, and requires maintaining spectral approximations to low-rank perturbations of \emph{elementwise (Hadamard) products}. This is the main challenge  and departure from recent IPM breakthroughs using inverse-maintenance, where low-rank updates to the slack matrix readily imply the same for the Hessian matrix.  
\end{abstract}

\newpage

\begin{toappendix}
\section{Initialization}\label{sec:init}
There exist standard techniques to transform a convex program to a form that has an easily obtainable strictly feasible point, see e.g.~\cite{ytm94}. We follow the initialization procedure presented by \cite{cls19} and \cite{jiang2020faster} and adapt to SOS optimization. Similar initialization lemma exists for WSOS optimization.

Let the matrix $P \in \R^{U \times L}$ and the operator $\Lambda: \R^U \to \R^{L \times L}$ that $\Lambda(s) = P^{\top} \diag(s) P$ be defined as in the interpolant basis paragraph of \Cref{sec_background}.
\begin{lemma}[Initialization]\label{lem:init}
Given an instance of $\eqref{SOS_program_primal_dual}$ that fulfills Slater's condition, and let $R$ be an upper bound on the $\ell_1$-norm of the primal feasible solutions, i.e.\ all primal feasible $x$ of \eqref{SOS_program_primal_dual} fulfill $\|x\|_1 \le R$, and let $\delta \in (0,1)$. We define
\begin{equation*}
\ov{A} = 
\begin{bmatrix}
A & 0 & \frac{1}{R}b - A\ov g^0 \\
\1_U^\top & 1 & 0
\end{bmatrix}
\in \R^{(m+1)\times (U+2)}
,\; 
\ov{b} = 
\begin{bmatrix}
\frac{1}{R}b \\
1 + \pr{\1_U}{ \ov g^0}
\end{bmatrix}
\in \R^{m + 1}, 
\text{and }
\ov{c} = \begin{bmatrix}
\frac{\delta}{\|c\|_\infty} c \\ 0 \\ 1
\end{bmatrix}
\in \R^{U + 2},
\end{equation*}
and let 
\begin{equation*}
    \ov{x}^0 = 
    \begin{bmatrix} 
        \ov g^0 \\ 1 \\ 1
    \end{bmatrix} 
    \in \R^{U + 2}, 
    \;
    \ov{y}^0 = 
    \begin{bmatrix}
    0_m \\
    -1 
    \end{bmatrix}
    \in \R^{m + 1}, 
    \text{and }
    \ov{s}^0 = 
    \begin{bmatrix}
    \1_U + \frac{\delta}{\|c\|_\infty} c \\
    1 \\ 
    1
    \end{bmatrix}
    \in \R^{U + 2},
\end{equation*}
where $\ov g^0 = g_{\Sigma^*}(\ov s^0_{[:U]}) \in \R^U$ for the gradient function $g_{\Sigma^*}(s) := \diag(P(P^\top \diag(s) P)^{-1} P^\top)$ that maps from $\R^U$ to $\R^U$. 
This defines the auxiliary primal-dual system 
\begin{equation}
  \label{auxiliary_SOS_program_primal_dual}
  \tag{Aux-SOS}
  \begin{aligned}
  \min \; &\pr{\ov c}{\ov x} \quad \\
  \ov A \ov x& = \ov b \\
  \ov x &\in \Sigma_{n,2d} \times \R_{\geq 0}^2\, ,\\
  \end{aligned}
  \quad\quad\quad
  \begin{aligned} 
  \max \; & \pr{\ov y}{\ov b} \\
  \ov A^\top \ov y + \ov s &= \ov c \\
  \ov s &\in \Sigma_{n,2d}^* \times \R_{\geq 0}^2\, . \\
  \end{aligned}
\end{equation}
Then $(\ov x^0, \ov y^0, \ov s^0)$ are feasible to the auxiliary system \eqref{auxiliary_SOS_program_primal_dual}. 

Further, under the canonical barrier (we use $\ov{a}_{i}$ to denote the $i$-th column of $\ov{A}$):
\begin{equation}\label{eq:init_barrier}
 \ov F_\eta(\ov{y}) = -\eta \pr{\ov{y}}{\ov{b}} - \log \det \Big(\Lambda\big((\ov{c} - \ov{A}^\top \ov{y})_{[:U]}\big)\Big) - \log (\ov{c}_{U+1} - \pr{\ov{a}_{U+1}}{\ov{y}}) - \log(\ov{c}_{U+2} - \pr{\ov{a}_{U+2}}{\ov{y}}),
\end{equation}
we have that $\|\ov g_{\eta^0}(\ov{y}^0)\|_{\ov H(\ov{y}^0)^{-1}} = 0$ for $\eta^0 = 1$. 

Further, for any solution $(\ov x, \ov y, \ov s)$ to \eqref{auxiliary_SOS_program_primal_dual} with duality gap $\le \delta^2$, its restriction $\wh{x} := \ov{x}_{[:U]}$ fulfills 
\begin{equation*}
\begin{aligned}
    \pr{c}{\wh{x}} &\le \min_{Ax = b, x \in \Sigma_{n,2d}} \pr{c}{x} +  \delta \cdot R \|c\|_\infty, \\
    \|A \wh{x} - b\|_1 &\le 8\delta L \cdot (LR \|A\|_\infty + \|b\|_1),\\
    \wh x &\in \Sigma_{n,2d}.
    \end{aligned}
\end{equation*}
\end{lemma}
\begin{proof}
Let $\operatorname{proj}$ denote the orthogonal projection matrix that maps a matrix onto its image space. Then note that 
\begin{equation}\label{eq:g_bound_s_0}
\begin{aligned}
    \|\ov{g}^0\|_1 &= \Big\|\diag\Big(P\Big(P^\top \diag(\1_U + \frac{\delta}{\|c\|_\infty} c) P\Big)^{-1} P^\top\Big)\Big\|_1 \\
    &= \tr \Big( \diag\big(1 + \frac{\delta}{\|c\|_\infty} c\big)^{-1} \cdot  \operatorname{proj}\big(\diag(\1_U + \frac{\delta}{\|c\|_\infty} c)^{1/2} \cdot P\big)\Big) \\
    & = (1 \pm \delta)L,
\end{aligned}
\end{equation}
where the last inequality follows as $\tr(\operatorname{proj}(P)) = \rk(P) = L$. 

Straightforward calculations show that $\ov{A}\ov{x}^0 = \ov{b}$ and $\ov{A}^\top \ov y^0 + \ov s^0 = \ov c$.
Further, note that $\ov{s}^0 > 0$ as $\delta < 1$ and therefore $\ov{s}^0 \in \R_+^{U+2} \subset \Sigma_{n,2d}^* \times \R_+^2$. The containment $\R_+^U \subset \R_{\geq 0}^U \subset \Sigma_{n,2d}^*$ is clear by the characterisation $\Sigma_{n,2d}^* = \{s : \Lambda(s) = P^\top \diag(s) P \succeq 0\}$.

Further $\ov{x}^0_{[:U]} = g_{\Sigma^*}(s^0_{[:U])}) = \diag(P(P^\top \diag(\ov{s}^0_{[:U]}) P)^{-1} P^\top) \in \Sigma_{n,2d}$ is e.g.~shown in Proposition 3.3.3 in \cite{renegarMathematicalViewInteriorPoint2001}. We make it explicit by showing that $\langle \ov{x}^0_{[:U]}, \hat s \rangle \ge 0$  for any $\hat s \in \Sigma_{n,2d}^*$. Indeed, 
\begin{equation*}
\begin{aligned}
    \langle \ov{x}^0_{[:U]}, \hat s \rangle &= \pr{\diag(P(P^\top \diag(\ov{s}^0_{[:U]}) P)^{-1} P^\top)}{\hat s} \\
    &= \tr\big( (P(P^\top \diag(\ov{s}^0_{[:U]}) P)^{-1} P^\top) \cdot \diag(\hat s)\big) \\
    &= \tr{\big((P^\top \diag(\ov{s}^0_{[:U]}) P)^{-1} (P^\top\diag(\hat s) P) \big)}\\
    &\ge 0,
\end{aligned}
\end{equation*}
where the last inequality follows as the argument in $\tr$ is a product of positive semi-definite matrices.

Define $\ov{s} = \ov{c} - \ov{A}^{\top} \ov{y}$. The corresponding barrier (as defined in \eqref{eq:init_barrier}) is
\begin{equation*}
\begin{aligned}
    \ov F_\eta(\ov{y}) &= -\eta \pr{\ov{y}}{\ov{b}} - \log \det \Big(\Lambda\big((\ov{c} - \ov{A}^\top \ov{y})_{[:U]}\big)\Big) - \log (\ov{c}_{U+1} - \pr{\ov{a}_{U+1}}{\ov{y}}) - \log(\ov{c}_{U+2} - \pr{\ov{a}_{U+2}}{\ov{y}}) \\
    &= -\eta \pr{\ov{y}}{\ov{b}} - \log \det \Big(\Lambda\big(\ov{s}_{[:U]}\big)\Big) - \log (\ov{c}_{U+1} - y_{m+1}) - \log(\ov{c}_{U+2} - \pr{\frac{1}{R}b - A\ov g^0}{\ov{y}_{[:m]}}), 
\end{aligned}
\end{equation*}
and gradient of the system is
\begin{equation*}
\begin{aligned}
    \ov g_\eta(\ov{y}) &= 
    \begin{bmatrix}
    -\frac{\eta}{R} b + A g_{\Sigma^*}(\ov{s}_{[:U]}) + \frac{1}{\ov{s}_{U+2}}(\frac{1}{R}b - A \ov g^0)
    \\ 
    -\eta(1 + \pr{\1_U}{\ov g^0}) + \pr{\1_U}{g_{\Sigma^*}(\ov s_{[:U]})} + \frac{1}{\ov{s}_{U+1}}
   \end{bmatrix}.
\end{aligned}
\end{equation*}
Simple algebraic calculations now show that for $\eta^0 = 1$ we have $\ov g_{\eta^0}(\ov{y}^0) = 0_{m+1}$, and so in particular $\|\ov g_{\eta^0}(\ov{y}^0)\|_{\ov H(\ov{y}^0)^{-1}} \le \epsilon_N$ as necessary in \Cref{lem:invariant_newton}.
It remains to show that near-optimal solutions to the modified problem correspond to near-optimal solutions to the original problem. 

Let $\OPT$ resp.~$\ov{\OPT}$ be the objective values of the original resp.~modified program:
\begin{equation*}
\OPT := \min_{Ax = b, x \in \Sigma_{n,2d}} \langle c,x \rangle, \qquad \ov{\OPT} := \min_{\ov{A}x = \ov{b}, \ov{x} \in \Sigma_{n,2d} \times \R^2 } \langle \ov{c},\ov{x} \rangle.
\end{equation*}
Given any optimal $x^* \in \R^U$ of the original primal of \eqref{SOS_program_primal_dual}, consider the following $\ov{x}^* \in \R^{U+2}$ fulfilling $\ov{A} \ov{x}^* = \ov{b}, \ov{x}^* \in \Sigma_{n,2d} \times \R_{\geq 0}^2$: 
\begin{equation*}
\ov{x}^* = \begin{bmatrix}
\frac{1}{R} x^* \\
\pr{\1_U}{\ov g^0} + 1 - \frac{1}{R} \pr{\1_U}{x^*} \\
0
\end{bmatrix}.
\end{equation*}
Note that indeed $\ov x^*_{U+1} \in \R_{\geq 0}$ as by \eqref{eq:g_bound_s_0} we have $x^*_{U+1} = \pr{\1_U}{\ov g^0} + 1 - \frac{1}{R} \pr{\1_U}{x^*} \ge (1-\delta )L + 1 - 1 \ge 0$ for $\delta < 1$ and by choice of $R$. 
Therefore
\begin{equation}
    \label{eq:opt_bound}
    \ov{\OPT} \le \pr{\ov c}{\ov{x}^*} =  
    \begin{bmatrix}
\frac{\delta}{\|c\|_\infty} c^\top & 0 & 1
\end{bmatrix} \cdot  
\begin{bmatrix}
\frac{1}{R} x^* \\
\pr{\1_U}{\ov g^0} + 1 - \frac{1}{R} \pr{\1_U}{x^*} \\
0
\end{bmatrix} = \frac{\delta \cdot \pr{c}{x^*}}{R\|c\|_\infty} = \frac{\delta \cdot \OPT}{R\|c\|_\infty}.
\end{equation}

Now given any feasible primal-dual solution $(\ov x, \ov y, \ov s)$ to \eqref{auxiliary_SOS_program_primal_dual} with duality gap $\le \delta^2$ we have
\begin{equation}
\label{eq:cx_bound}
    \frac{\delta}{\|c\|_\infty} \pr{c}{\ov x_{[:U]}} \le \frac{\delta}{\|c\|_\infty} \pr{c}{\ov x_{[:U]}} + \ov x_{U+2} = \pr{\ov c}{\ov x}  \le \ov \OPT + \delta^2 \le \frac{\delta}{R\|c\|_\infty} \OPT + \delta^2,
\end{equation}
where the first inequality uses $\ov x_{U+2} \in \R_{\geq 0}$ and the last inequality follows from \eqref{eq:opt_bound}. So, for $\wh x := R \cdot \ov{x}_{[:U]}$, using \eqref{eq:cx_bound} we have that 
\begin{equation*}
    \pr{c}{\wh{x}} = R \pr{c}{\ov{x}_{[:U]}} = \frac{R\|c\|_\infty}{\delta} \frac{\delta}{\|c\|_\infty} \pr{c}{\ov{x}_{[:U]}} \le \frac{R\|c\|_\infty}{\delta} \Bigl(\frac{\delta}{R\|c\|_\infty} \OPT + \delta^2\Bigr) = \OPT + \delta R\|c\|_\infty.
\end{equation*}
Recall that $\R_{\geq 0}^U \subset \Sigma_{n,2d}^*$ and therefore $\Sigma_{n,2d} \subset \R_{\geq 0}^U$. In particular $\ov x_{[:U]} \ge 0$ and so using the equality $\pr{\1_U}{\ov x_{[:U]}} + \ov x_{U+1} = 1 + \pr{\1_U}{ \ov g^0}$ (which follows from $\ov{A}\ov{x} = \ov{b}$ since $(\ov{x}, \ov{y}, \ov{s})$ is feasible) and $\ov x_{U+1} \ge 0$ we can further bound
\begin{equation*}
   \|\ov x_{[:U]}\|_1 = \pr{\1_U}{\ov x_{[:U]}} \le 1 + \pr{\1_U}{\ov g^0} \le 1 +(1 + \delta)L \le 2 L\, ,
\end{equation*}
where the penultimate inequality used \eqref{eq:g_bound_s_0}. So we can further bound
\begin{equation}
\label{eq:bound_hat_x}
\frac{1}{\|c\|_\infty} \pr{c}{\ov x_{[:U]}} \ge - \frac{1}{\|c\|_\infty} \|c\|_\infty \|\ov x_{[:U}\|_1 
\ge - 2L.
\end{equation}
Therefore using \eqref{eq:cx_bound} and \eqref{eq:bound_hat_x} we get that
\begin{equation}\label{eq:bound_ov_x_U+2}
\ov x_{U+2} \le \frac{\delta}{R\|c\|_\infty} \OPT + \delta^2 - \frac{\delta}{\|c\|_\infty} \pr{c}{\ov x_{[:U]}} \le \frac{\delta}{R\|c\|_\infty} \OPT + \delta^2 + 2\delta L \le \delta + \delta^2 + 2\delta L \le 4\delta L .
\end{equation}
where the penultimate inequality used $\OPT \le R\|c\|_\infty$ and the last inequality uses $\delta \le 1$.

Feasibility of $\ov x$ to $\ov A \ov x = \ov b$ shows that $A \ov x_{[:U]} + (\frac{1}{R}b - A\ov{g}^0)\ov{x}_{U+2} = \frac{1}{R} b$
and we therefore have 
\begin{equation*}
\begin{aligned}
    \|A\wh{x} - b \|_1 &= \|RA\ov{x}_{[:U]} - b\|_1 = \|(b - RA \ov{g}^0)\ov x_{U+2}\|_1 \le 4\delta L\cdot(R\|A\|_\infty \|\ov g^0\|_1 + \|b\|_1) \\
    & \le 8\delta L\cdot (LR \|A\|_\infty + \|b\|_1).
    \end{aligned}
\end{equation*}
where the penultimate inequality used \eqref{eq:bound_ov_x_U+2}, and the last inequality follows from \eqref{eq:g_bound_s_0}.
\end{proof}

It is worth noting that the system $\ov{A}\ov x = \ov b, \ov x \in \Sigma_{n,2d} \times \R_{\geq 0}^2$ is not an instance of $\eqref{SOS_program_primal_dual}$ anymore. This is in contrast for similar initialization techniques for Linear Programming \cite{cls19} or semi-definite Programming \cite{jiang2020faster}. Nonetheless it is not hard to see that optimization over product cones as in this case $\Sigma_{n,2d} \times \R_{\geq 0}^2$ can be done efficiently if the optimization over the factors can be done efficiently, see e.g.~the application in  \cite{lsz19}.
Alternatively note that the dual cone $(\Sigma_{n,2d} \times \R_{\geq 0}^2)^* = \Sigma_{n,2d}^* \times \R_{\geq 0}^2$. For $s \in \Sigma_{n,2d}^*$ we used the barrier $F(s) = - \log(\det(P^\top\diag(s)P))$ throughout the paper. This barrier could easily extend via setting
\begin{equation*}
    \ov{P} := 
    \begin{bmatrix}
    P & 0 & 0 \\
    0 & 1 & 0 \\
    0 & 0 & 1 \\
    \end{bmatrix}
\end{equation*}
and setting the barrier for $\ov s \in \Sigma_{n,2d}^* \times \R_{\geq 0}^2$ as 
\begin{equation*}
   \ov{F}(\ov s) = - \log\det(\ov{P}^\top \diag(\ov{s}) \ov{P}) = - \log(\det(P^\top\diag(\ov{s}_{[:U]})P)) - \log \ov{s}_{U+1} - \log \ov{s}_{U+2},
\end{equation*}
which recovers the standard barrier for the product cone. Our algorithm \Cref{alg:barrier} can now equally be run with $\ov{F}$ instead of $F$, where then $P$ is replaced by $\ov{P}$.
 \section{Discussion on representation}
In this paper we consider the dual formulation of \eqref{SOS_program_primal_dual}, where $A \in \R^{m \times U}$, $b \in \R^m$, $c \in \R^U$:
\begin{equation}\label{eq:sos_dual}
    \begin{aligned}
  \max \; & \pr{y}{b} \\
  A^\top y + s &= c \\
  s &\in \Sigma_{n,2d}^*\, .
 \end{aligned}
\end{equation}

Here $m$ is the number of constraints, and w.l.o.g. we assume $m \leq U$.

We remark that \cite{py19} instead consider the following formulation:

\begin{align*} 
  \min \; & \pr{c}{x} \\
  A x &= b \\
  x &\in \Sigma_{n,2d}^*\, .
\end{align*}

These two formulations are in fact equivalent, because Eq.~\eqref{eq:sos_dual} is equivalent to

\begin{align*}
  \min \; &\pr{\hat c}{x} \quad \\
  \hat Ax& = \hat b \\
  x &\in \Sigma_{n,2d}^*\,.
\end{align*}

Here we define $\hat{A} \in \R^{(U-m) \times U}$ to be a matrix that satisfies $\ker(\hat A) = \im(A^\top)$, i.e., the rows of $\hat{A}$ are the orthogonal complement of the rows of $A$. Further we define $\hat c \in \R^U$ to be any vector that satisfies $A \hat c = b$. Finally, we define $\hat b := \hat A c \in \R^{U-m}$.  

\subsection{SOS under the monomial basis}\label{sec:other_basis}

In \cite{py19}, the advantages and disadvantages of three different bases (monomial basis, Chebyshev basis, and interpolant basis) are discussed. Further, explicit expressions of $\Lambda$ (defined in \Cref{prop:nesterov_dual_SOS}) are given for all three bases. Apart from the advantage of its computational efficiency, \cite{py19} chooses interpolant basis in their algorithm also because the interpolant basis representation is numerically stable, and this is required for practical algorithms. As we mainly focus on the theoretical running time of SOS algorithms, in this section we further justify the choice of the interpolant basis in our algorithm. We want to add to the exposition of \cite{py19} and stress that the standard monomial basis does not seem suitable even for theoretical algorithms. It is unclear whether in the monomial basis an amortized runtime faster than the naive $O(U^\omega)$ can be achieved. (Similar argument holds for the Chebyshev basis.) 

The Hessian in the monomial basis is given by
\begin{equation*}
H_{ij}(s) = \sum_{a + b = i, k + \ell = j} [\Lambda(s)^{-1}]_{ak}[\Lambda(s)^{-1}]_{b\ell}.
\end{equation*}
It is unclear how low-rank updates techniques could be applied here as low rank updates to $\Lambda$ do not translate to low-rank updates to the Hessian $H$. Further, the structure of $\Lambda$ itself in the multivariate case is far less understood than its counterpart in the interpolant basis, which we are going to elaborate on in the remainder of this section.

For $n$ variables and degree $d$ the monomial basis elements correspond to the terms $x_1^{\alpha_1} \cdot \ldots \cdot x_n^{\alpha_n}$ for $\alpha = (\alpha_1, \ldots, \alpha_n) \in \mathbb{N}^n$, $\|\alpha\|_1 \le d$. When choosing both $\mathbf{p}$ and $\mathbf{q}$ to be the monomial basis, $\Lambda$ has a special structure. For any basis elements $p_1, p_2 \in \mathbf{p}$ we have that $p_1 p_2$ itself is a monomial in $\mathcal V_{n,2d}$. As such, the coefficients $\lambda_{ij} \in \R^U$ have the special form that $\lambda_{ij}$ is zero everywhere but in the coordinate that corresponds to the element in $\mathbf{q}$ equalling $p_1p_2$. As $\Lambda$ is also uniquely defined by the images of the elements in $\mathbf q$, we can write $E_u \in \R^{L \times L}$ as $E_u := \Lambda(e_u)$ where $e_u$ is the vector that is zero everywhere but in position $u$ for some $u \in [U]$. Let $q_u \in \mathbf{q}$ be the associated basis polynomial. Then we see that $(E_u)_{ij} = 1_{[p_ip_j = q_u]}$. It also follows that every matrix $S \in \im(\Lambda)$ has at most $U$ different entries and each entry is uniquely defined by the corresponding basis element in $\mathbf{q}$. While it is not known how this special structure could be exploited in general, a speedup is known for the univariate case $n = 1$. Here, $\im(\Lambda)$ are all Hankel matrices: We have $\mathbf{q} = \{1, x, \ldots, x^{2d}\}$ and $\mathbf{p} = \{1, x, \ldots, x^d\}$ so for any vector $u \in \R^{2d+1}$ we have that 

\begin{equation*}
    \Lambda(v) = 
    \begin{bmatrix}
  v_{0} & v_{1} & v_{2} & \ldots & \ldots  &v_{d}  \\
  v_{1} & v_2 &  &  & &\vdots \\
  v_{2} &  &  & & & \vdots \\ 
 \vdots & & & &  & v_{2d-2}\\
 \vdots & &  & & v_{2d-2}&  v_{2d-1} \\
v_{d} &  \ldots & \ldots & v_{2d-2} & v_{2d-1} & v_{2d}
\end{bmatrix}.
\end{equation*}
These highly structured matrices are known to be invertible in time $\tilde O(d^2)$ (see e.g.\ \cite{Pan2001}).

As mentioned above even the bivariate case becomes far more complicated. Let $n = 2$ and $d = 2$, and pick the ordered bases as 
\begin{equation*}
\mathbf q = \{1, x, y, x^2, xy, y^2, x^3, x^2y, xy^2, y^3, x^4, xy^3, x^2y^2, xy^3, y^4\}\, , ~~
\mathbf p = \{1, x, y, x^2, xy, y^2\}\, .
\end{equation*}
 Then for $v \in \R^{n+ 2d \choose 2d} = \R^{15}$ we get for $\mathbf p \mathbf p^\top$ and the corresponding matrix $\Lambda$ that

\begin{equation*}
    \mathbf p \mathbf p^\top =
    \begin{bmatrix}
    1 & x & y & x^2 & xy & y^2\\
    x & x^2 & xy & x^3 & x^2y & xy^2\\
    y & xy & y^2 & x^2y & xy^2 & y^3 \\
    x^2 & x^3 & x^2y & x^4 & x^3y & x^2y^2 \\
    xy & x^2y & xy^2 & x^3y & x^2y^2 & xy^3 \\
    y^2 & xy^2 & y^3 & x^2y^2 & xy^3 & y^4 
   \end{bmatrix}, \quad 
    \Lambda(v) = 
    \begin{bmatrix}
    v_0 & v_1 & v_2 & v_3 & v_4 & v_5 \\
    v_1 & v_3 & v_4 & v_6 & v_7 & v_8 \\
    v_2 & v_4 & v_5 & v_7 & v_8 & v_9 \\
    v_3 & v_6 & v_7 & v_{10} & v_{11} & v_7 \\
    v_4 & v_7 & v_8 & v_{11} & v_{12} & v_{13} \\
    v_5 & v_8 & v_9 & v_7 & v_{13} & v_{14} \\
   \end{bmatrix}.
\end{equation*}
While $\Lambda(v)$ still has some structure it is unclear how $(\Lambda(v))^{-1}$ could be computed more efficiently than in matrix multiplication time.

\section{Proof of amortization lemma}\label{sec:proof_amortize}
We include the proof of Lemma~\ref{lem:general_amortize_tool} for completeness. The main difference between this proof and that of \cite{hjst21} is that we cut off at $U/L$ instead of $L$.

Our proof makes use of the following two facts about $\omega$ and $\alpha$ (Lemma~A.4 and Lemma~A.5 of \cite{cls19}).

\begin{fact}[Relation of $\omega$ and $\alpha$]\label{fact:bound_omega_alpha}
$\omega \leq 3 - \alpha$.
\end{fact}

\begin{fact}[Upper bound of $\Tmat(n,n,r)$]\label{fact:approx_Tmat_r_leq_n}
For any $r \leq n$, we have that
$\Tmat(n,n,r) \leq n^{2 + o(1)} + r^{\frac{\omega - 2}{1 - \alpha}} \cdot n^{2 - \frac{\alpha (\omega-2)}{(1 - \alpha)} + o(1)}$.
\end{fact}

\begin{lemma}[Restatement of Lemma~\ref{lem:general_amortize_tool}]
Let $t$ denote the total number of iterations. Let $r_i \in [L]$ be the rank for the $i$-th iteration for $i \in [t]$. Assume $r_i$ satisfies the following condition:
for any vector $g \in \R_+^L$ which is non-increasing, we have 
\begin{align*}
\sum_{i=1}^t r_i \cdot g_{r_i} \leq O(t \cdot \|g\|_2).
\end{align*}

If the cost in the $i$-th iteration is $O(\Tmat(U, U, \min\{L r_i, U\}))$, when $\alpha \geq 5 - 2 \omega$, the amortized cost per iteration is 
\begin{equation*}
U^{2 + o(1)} + U^{\omega - 1/2 + o(1)} \cdot L^{1/2}.
\end{equation*}
\end{lemma}
\begin{proof}
For $r_i$ that satisfies $r_i \leq U/L$, we have
\begin{equation}\label{eq:amortization_eq1}
\begin{aligned}
    \Tmat(U, U, L r_i) \leq &~ U^{2 + o(1)} + (L r_i)^{\frac{\omega-2}{1-\alpha}} \cdot U^{2 - \frac{\alpha(\omega-2)}{1-\alpha} + o(1)} \\
    = &~ U^{2 + o(1)} + U^{2 - \frac{\alpha(\omega-2)}{1-\alpha} + o(1)} \cdot L^{\frac{\omega-2}{1-\alpha}} \cdot r_i^{\frac{\omega-2}{1-\alpha}},
\end{aligned}
\end{equation}
where the first step follows from Fact~\ref{fact:approx_Tmat_r_leq_n}.

Define a sequence $g \in \R_+^L$ such that for $r \in [L]$, 
\begin{equation*}
\begin{aligned}
    g_r =
    \begin{cases}
    r^{\frac{\omega-2}{1-\alpha}-1} & \text{if~} r \leq U/L, \\
    (U/L)^{\frac{\omega - 2}{1 - \alpha}} \cdot r^{-1} & \text{if~} r > U/L.
    \end{cases}
\end{aligned}
\end{equation*}
Note that $g$ is non-increasing because $\frac{\omega-2}{1-\alpha} \leq 1$ (Fact~\ref{fact:bound_omega_alpha}). Then using the condition in the lemma statement, we have
\begin{equation}\label{eq:amortization_eq2}
\begin{aligned}
    \sum_{i=1}^t \min\{r_i^{\frac{\omega-2}{1-\alpha}}, (U/L)^{\frac{\omega-2}{1-\alpha}}\} = &~ \sum_{i=1}^t r_i \cdot g_{r_i} \\
    \leq &~ t \cdot \|g\|_2 \\
    \leq &~ t \cdot \Big(\int_{x=1}^{U/L} x^{\frac{2(\omega-2)}{1-\alpha}-2} \mathrm{d}x + (U/L)^{\frac{2(\omega - 2)}{1 - \alpha}} \cdot \int_{x=U/L}^L x^{-2} \mathrm{d}x \Big)^{1/2} \\
    \leq &~ t \cdot \Big( c \cdot (U/L)^{\frac{2(\omega-2)}{1-\alpha}-1} + (U/L)^{2(\frac{\omega - 2}{1 - \alpha})} \cdot (U/L)^{-1}\Big)^{1/2} \\
    = &~ t \cdot O((U/L)^{\frac{(\omega-2)}{1-\alpha}-1/2}),
\end{aligned}
\end{equation}
where the first step follows from the definition of $g \in \R^L$, the second step follows from the assumption $\sum_{t=1}^t r_i \cdot g_{r_i} \leq t \cdot \|g\|_2$ in the lemma statement,  
the third step follows from upper bounding the $\ell_2$ norm $\|g\|_2^2 = \sum_{r=1}^L g_r^2$, and the fourth step follows $\frac{2(\omega-2)}{1-\alpha} \geq 1$ when $\alpha \geq 5 - 2 \omega$, so the integral $\int_{x=1}^{U/L} x^{\frac{2(\omega-2)}{1-\alpha}-2} \mathrm{d}x = c \cdot x^{\frac{2(\omega-2)}{1-\alpha}-1}\big|_{1}^{U/L} = O\big((U/L)^{\frac{2(\omega-2)}{1-\alpha}-1}\big)$ where $c:= 1/(\frac{2(\omega-2)}{1-\alpha}-1)$. 

Thus we have
\begin{equation*}
\begin{aligned}
    \sum_{t=1}^t \Tmat(U, U, \min\{L r_i, U\})
    \leq & ~ \sum_{t=1}^t \Big( U^{2 + o(1)} + U^{2 - \frac{\alpha(\omega-2)}{1-\alpha} + o(1)} \cdot L^{\frac{\omega-2}{1-\alpha}} \cdot \min\{r_i^{\frac{\omega-2}{1-\alpha}}, (U/L)^{\frac{\omega-2}{1-\alpha}} \} \Big) \\
    = &~ t \cdot U^{2 + o(1)} + U^{2 - \frac{\alpha(\omega-2)}{1-\alpha} + o(1)} \cdot L^{\frac{\omega-2}{1-\alpha}} \cdot \sum_{t=1}^t \min\{r_i^{\frac{\omega-2}{1-\alpha}}, (U/L)^{\frac{\omega-2}{1-\alpha}}\} \\ 
    \leq &~ t \cdot U^{2 + o(1)} + U^{2 - \frac{\alpha(\omega-2)}{1-\alpha} + o(1)} \cdot L^{\frac{\omega-2}{1-\alpha}} \cdot t \cdot (U/L)^{\frac{(\omega-2)}{1-\alpha}-1/2} \\
    = &~ t \cdot (U^{2 + o(1)} + U^{\omega - 1/2 + o(1)} \cdot L^{1/2}),
\end{aligned}
\end{equation*}
where the first step follows from Eq.~\eqref{eq:amortization_eq1} and $\Tmat(U,U,U) = U^{\omega} = U^{2 - \frac{\alpha(\omega-2)}{1-\alpha}} \cdot L^{\frac{\omega-2}{1-\alpha}} \cdot (U/L)^{\frac{\omega-2}{1-\alpha}}$, the second step follows from moving summation inside, the third step follows from Eq.~\eqref{eq:amortization_eq2}, and the last step follows from adding the terms together.
\end{proof}

\end{toappendix}

\section{Introduction}

Polynomial optimization is a fundamental problem in many areas of applied mathematics, operations research, and  theoretical computer science, including combinatorial optimization \cite{BS11, LasserreNing, BHKKMP19}, statistical estimation \cite{HopkinsKPRSS17, Hopkins018}, experimental design \cite{Papp12}, control theory \cite{HessHLP16}, signal processing \cite{RohDV07}, power systems engineering \cite{GPF},  discrete geometry \cite{BV08, BBBCGKS09} and computational algebraic geometry \cite{AG09}. In the most basic formulation, we are given a collection of $k$ real $n$-variate polynomials $g_1,\cdots, g_k$ and an objective function $f: \R^n \to \R$, and the goal is to minimize $f$ over the set 
$\cS := \{t \in \R^n \mid \forall i \in \{1,\cdots, k\}: g_i(t) \geq 0\}$, that is, to find  
\begin{equation}\label{eq_poly_opt}
\inf_{t\in \R^n} \{f(t) \mid t\in \cS \},
\end{equation}
which is equivalent to checking polynomial nonnegativity $\sup_{c\in \R} \{c \mid f(t) - c \geq 0, ~\forall t \in \cS\}$. This is then equivalent to computing $\sup_{c \in \R} \{c \mid f-c \in \mathcal{K}(\cS)\}$, where $\mathcal{K}(\cS)$ denotes the convex cone of all polynomials of degree at most $\deg(f)$ that are non-negative on the set $\cS$. This is an instance of the more general conic programming: 

\begin{equation}
    \label{conic_program_primal}
    \begin{aligned}
    \min_{x \in \R^N} \;  &c^\top x \quad \\
    Ax& =b \\
    x &\in \K,\ \\
    \end{aligned}
    \end{equation}

where $\K \subset \R^N$ is some convex  cone\footnote{A subset $\K \subset \R^N$ is a convex cone if $\forall \; x,y \in \K$ and $\alpha, \beta\in\R_+$, $\alpha x+ \beta y \in \K$.}. The conic optimization problem over the cone $\mathcal{K}(\cS)$ is intractable in general because there is no simple  characterization of $\mathcal{K}(\cS)$. Nevertheless, there always exists an increasing family of convex cones of \emph{weighted sum-of-squares} polynomials that converges to any such cone $\mathcal{K}(\cS)$.

We first introduce the notion of \emph{sum-of-squares} (SOS) polynomials: 
Denoting by $\V_{n,d}$ the vector space of all $n$-variate polynomials of (total) degree at most $d$, a polynomial $p \in \V_{n,2d}$ is said to be \emph{sum-of-squares} (SOS) if it can be written as a finite sum of square polynomials, i.e., there exist $q_1,\cdots, q_\ell$ such that $p = \sum_{i=1}^\ell q_i^2$. The set $\Sigma_{n,2d}$ of SOS polynomials of degree at most $2d$ is a (proper) cone contained in $\V_{n,2d}$, of dimension $U := \dim(\V_{n,2d}) = {n+2d \choose d}$, as the vector space $\V_{n,2d}$ is isomorphic to $\R^{U}$. If $p$ can be written as $p = \sum_{i=1}^k g_i s_i$ for $s_1 \in \Sigma_{n,2d_1}, \cdots, \Sigma_{n,2d_k}$
and $k$ nonzero polynomials $\boldf := (f_1,\cdots, f_k)$, 
then it is said to be \emph{weighted sum-of-squares} (WSOS). 

\emph{Putinar's Positivstellensatz} \cite{PUTINAR99} states that under mild conditions, any polynomial $p$ that is non-negative on $\cS$ can be written as a WSOS polynomial $\sum_{i=1}^k g_i s_i$, albeit with (potentially) \emph{unbounded degree} $s_i$'s. In WSOS optimization we consider sum-of-squares polynomials $s_i$ with bounded degree, so the hierarchy of WSOS optimization with increasing degree (known as the Lasserre hierarchy) can be viewed as a tool for approximating general polynomial optimization. For more details of this approximation scheme for polynomial optimization, we refer the readers to the matextbooks \cite{l15, bpt13}.

This paper concerns algorithms for \emph{(W)SOS optimization}, which is the conic optimization program  \eqref{conic_program_primal} where the underlying cone $\K$ is the (W)SOS cone: 

\begin{equation}
    \label{SOS_program_primal}
    \begin{aligned}
    \min_{x \in \R^U} \; & c^\top x \quad \\
    Ax& =b \\
    x &\in \Sigma_{n,2d},\ \\
    \end{aligned}
    \end{equation}
where $x\in \vntd$ is the vector of coefficients which encodes the polynomial. Henceforth, we focus on the case where $\K = \Sigma_{n,2d}$ is the SOS cone, and discuss how to extend our algorithm for SOS optimization to WSOS in Section \ref{sec:wsos}.

The computational complexity of solving 
Problem \ref{SOS_program_primal} naturally depends on the dimensions 
\begin{equation}\label{eq_L_U} 
L := \dim(\V_{n,d}) = {n+d \choose d} \;\;\;\;\;, \;\;\;\; U := \dim(\V_{n,2d}) = {n+2d \choose 2d} 
\end{equation}
of the underlying vector spaces
(Note that $L \leq U\leq L^2$). We now turn to explain the previous  approaches for SOS optimization solvers.  

\paragraph*{SOS Optimization as SDPs}
A fundamental fact is that the \emph{dual} SOS cone is a \emph{slice of the SDP cone} \cite{nesterov2000squared}. More formally, for any \emph{fixed bases} $\p = (p_1, p_2, \cdots, p_L)$ and $\q = (q_1, q_2, \cdots, q_U)$ to $\vnd$ and $\vntd$ respectively, 
there exists a unique linear mapping $\Lambda: \R^U \to \R^{L \times L}$ satisfying 
\begin{equation} \label{eq_Lambda}
\Lambda(\q(t)) = \p(t) \p(t)^\top, ~~\forall t \in \R^n.
\end{equation}
Here we define $\p(t) = (p_1(t), p_2(t), \cdots, p_L(t))^{\top}$ and $\q(t) = (q_1(t), q_2(t), \cdots, q_U(t))^{\top}$.
An equivalent way to view the definition of $\Lambda$ in \eqref{eq_Lambda} is as follows: For polynomials $p_i,p_j \in \p$ there are unique coefficients $\lambda_{iju}$ such that $p_ip_j = \sum_{u \in U} \lambda_{iju} q_u$. These $\lambda_{iju}$ define the mapping $\Lambda$ unambiguously.

This in turn implies that a polynomial $s \in \vntd$ (we view $s$ as a vector in $\R^U$ that corresponds to its coefficients over the basis $\q$) is in the dual SOS cone $\Sigma_{n,2d}^*$ if and only if $\Lambda(s)$ is a \emph{positive semidefinite} (PSD) matrix
(proved by \cite{nesterov2000squared}, see Theorem \ref{prop:nesterov_dual_SOS} for details). 
As \cite{py19} recently observed, the choice of the bases $\p,\q$ crucially affects the complexity of the optimization problem, more on this below. 

Equation \eqref{eq_Lambda} implies the well-known fact that optimization over SOS polynomials \eqref{SOS_program_primal} can be reduced to semidefinite programming
\begin{equation}
\label{eq_SDP}\tag{SDP}
    \min_{X \succeq 0} \{ \pr{C}{X} \mid \tr(A_iX) = b_i, ~ \forall i \in [m]\},
\end{equation}
and can thus be solved using off-the-shelve SDP solvers. However, despite recent breakthroughs on the runtime of general SDP solvers via \emph{interior-point methods} (IPMs) \cite{jiang2020faster, hjst21}, the SDP reformulation \eqref{eq_SDP} of \eqref{SOS_program_primal} does not scale well for moderately large degrees, i.e., 
whenever $U \ll L^2$ in \eqref{eq_L_U}. This is because the SDP reformulation always incurs a factor of at least $L^2$, even when $U\ll L^2$, as this is the SDP variable size (the PSD matrix $X$ has size $L \times L$). Indeed, for current fast-matrix-multiplication time $\omega \approx 2.37$ \cite{l14, aw21}, 
the running time of state-of-the-art SDP solvers \cite{jiang2020faster, hjst21} for SOS optimization (Problem \ref{SOS_program_primal}) is\footnote{We use $\wt{O}(\cdot)$ to hide $U^{o(1)}$ and $\log (1/\delta)$ factors.}
\begin{equation}\label{eq_SDP_runtim}
 \wt{O}\left( L^{0.5} \cdot \min\{ UL^2 + U^{2.37}, ~ L^{4.24} \} \right ).
\end{equation}

An alternative approach is to solve Problem \ref{SOS_program_primal} directly by designing an \emph{ad-hoc} IPM for the dual SOS cone, avoiding the blowup in the SDP reformulation. This was exactly the motivation of \cite{py19}.
In more detail, all aformentioned SDP solvers are  based on IPMs \cite{nesterov1994interior}, which iteratively minimize the original objective function plus a barrier function via Newton steps. 
When applied to the SOS Problem \eqref{SOS_program_primal}, 
\emph{the choice of the specific bases $\p,\q$ crucially affects the structure of the (Hessian of the) barrier function $F(s)=F(\Lambda(s))$}, and hence the cost-per-iteration of the IPM. As such, choosing a ``good'' and efficient basis is key to a fast algorithm for \eqref{SOS_program_primal}. One of the main contribution of \cite{py19} is an efficient basis for the SOS cone, which efficiently scales to multivariate SOS, yielding an IPM whose total runtime is 
\begin{equation}\label{eq_PY_runtim}
\wt{O}\left( L^{0.5}U^\omega \right) \approx \wt{O}\left( L^{0.5}U^{2.37} \right).
\end{equation}

Our main result is a polynomially faster IPM for Problem \ref{SOS_program_primal}: \begin{theorem}[Main Result, Informal version of \Cref{thm:correct}]\label{thm_main_informal}
With current FMM exponent, there is an algorithm for solving Problem \eqref{SOS_program_primal}, whose total running time is $\wt{O}\left(  
LU^{1.87}
\right).$
\end{theorem}

Indeed, this runtime is polynomially faster than \eqref{eq_PY_runtim} and \eqref{eq_SDP_runtim}, as shown in \Cref{fig:runtimes}.
We now turn to elaborate on the technical approach for proving Theorem~\ref{thm_main_informal}.
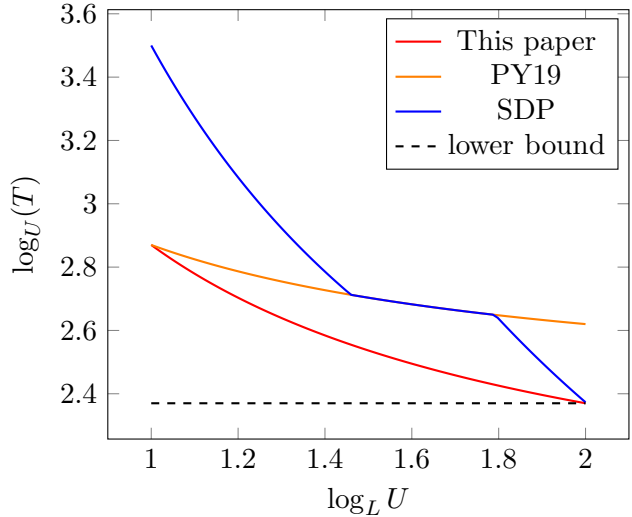
\begin{figure}[!ht]
\centering
\begin{tikzpicture}
    \begin{axis}[every axis plot/.append style={thick},xlabel=$\log_L U$, ylabel=$\log_U(T)$]

    \addplot [color=red, domain=1:2, samples=90] {1/(2*x) + max(2,1.87 + 1/(2*x)};
    \addlegendentry{This paper}

    \addplot [color=orange, domain=1:2, samples=90] {1/(2*x) + 2.37};
    \addlegendentry{PY19}

    \addplot [color=blue,domain=1:2, samples=90] {1/(2*x) + min(max(2.37,1+2/x),4.246/x)};
    \addlegendentry{SDP}

    \addplot [color=black, style=dashed, domain=1:2, samples=90] {2.37};
    \addlegendentry{lower bound}
    
    \end{axis}
\end{tikzpicture}
\caption{Overview of current running times of recent solvers for SOS. The lower bound bound stems from solving a linear system in $U$ variables, i.e., $T = \Omega(U^{\omega})$ where $\omega \approx 2.37$.}\label{fig:runtimes}.
\end{figure}

\paragraph*{Faster IPMs via Inverse-Maintenance} Interior-Point Methods (IPMs \cite{karmarkar1984new, renegarMathematicalViewInteriorPoint2001}) are a powerful class of second-order optimization 
algorithms for convex optimization, which essentially reduce a conic optimization problem \eqref{conic_program_primal} to solving \emph{a sequence of slowly-changing linear systems} (via Newton steps).
Since their discovery in the mid 80's,  IPMs have emerged as the ``gold-standard'' of convex optimization, as they are known to converge fast in \emph{both theory and practice} \cite{s87}.  The main computational cost of IPMs is computing, in each iteration,   
the inverse of the Hessian of the underlying barrier function $F(s)= F(\Lambda(s))$, which naively costs at least $U^\omega$ time per iteration for the SOS optimization problem \cite{py19}. A recent influential line of work \cite{cls19, jiang2020faster}, inspired by \cite{v89_lp}'s seminal work, has demonstrated that \emph{dynamically maintaining} the inverse of the Hessian matrix under \emph{low-rank} updates using clever \emph{data structures},   
can lead to much cheaper \emph{cost-per-iteration}. 
All of these results rely on a careful combination of dynamic 
data structures with the geometry (e.g., spectral approximation) of the underlying optimization method and barrier function. 
This paper extends this line of work to SOS optimization. 

\paragraph*{Our Techniques} 
We follow the framework of \cite{py19} which chooses the \emph{polynomial interpolant basis} representation and the corresponding linear operator $\Lambda: \R^U \to \R^{L \times L}$ is $\Lambda(s) = P^{\top} \diag(s) P$, where $P \in \R^{U\times L}$ is the matrix whose entries are the evaluation of the Lagrange interpolation polynomials, through some unisolvent\footnote{Any set of points in $\R^n$ for which the evaluation of a polynomial in $\vnd$ on these points uniquely defines the polynomial.} set of points in $\vnd$ (see \Cref{sec_background} for a formal definition). This basis induces the aforementioned convenient form of $\Lambda$, and generalizes to the multivariate case.  
The Hessian of the barrier function 
$F(s) = - \log \det(\Lambda(s))$ 
is given by 
\begin{equation*}
    H(s) = \big( P(P^{\top} \diag(s) P)^{-1} P^{\top} \big)^{\circ 2} \in \R^{U \times U},
\end{equation*}
where $A\circ B$ denotes the element-wise (Hadamard) product of two matrices. The main bottleneck of each iteration of IPMs is to compute the Hessian inverse $H(x)^{-1}$ of the Newton step, which na\"ively takes $O(U^{\omega})$ time. 

In IPM theory, it has long been known that it suffices to compute a spectral approximation of the Hessian. We follow the ``lazy update'' framework in recent developments of LP and SDP solvers \cite{cls19, jiang2020faster}, which batches together low-rank updates to $M:= P(P^{\top} \diag(s) P)^{-1} P^{\top}$, where $\rk(M) = L$. In each iteration, we can compute a spectral approximation $M^{\new} = M + U V^{\top}$, where $U,V$ are low rank matrices with size $U \times r$ where $r \ll U$ is chosen to optimize the runtime. Since $\wt{M} \approx M$ implies that $\wt{M}^{\circ 2} \approx M^{\circ 2}$, this also gives a spectral approximation of the Hessian.

The main challenge here, compared to previous LP and SDP solvers \cite{v89_lp, ls14, cls19, jiang2020faster, hjst21}, is that low-rank updates to $M$ do not readily translate to a low-rank update to $(M^{\circ 2})^{-1}$, since Hadamard-products can \emph{increase} the rank $\rk(A\circ B) \leq \rk(A)\cdot \rk(B)$, in contrast to standard matrix multiplication which does not increase the rank $\rk(A B) \leq \max\{\rk(A), \rk(B)\}$. This means that we cannot directly apply 
Woodbury's identity to efficiently update the inverse of the Hessian, which is the common approach in all aforementioned works. Instead, we employ the following property which relates rank-one Hadamard-product perturbations to standard matrix products 
\[
M \circ (u \cdot v^{\top}) = \diag(u) \cdot M \cdot \diag(v),
\]
which means that we can translate the rank-$r$ update of $M$ into a rank-$Lr$ update of $M^{\circ 2}$ for $r \le L$. With some further calculations, applying Woodbury's identity on the resulting matrix, implies that we can compute $((M^{\new})^{\circ 2})^{-1}$ in time 
\[
O\big(\Tmat(U,U,Lr)\big),
\] 
which is never worse than $\Tmat(U,U,U) = U^{\omega}$ as long as $r \le U/L$.
Modifying the amortization tools of \cite{jiang2020faster} and \cite{hjst21}, combined with basic spectral theory for Hadamard products,
we show that our amortized cost per iteration is bounded by
\[ O\big(U^2 + U^{\omega - 1/2} \cdot L^{1/2}\big),\]
which becomes $O\big(U^{2} + U^{1.87} L^{0.5}\big)$ if we plug in the current matrix multiplication exponent.

 \section{Preliminaries}
In this section we provide the definitions and the tools that we will use.
For any integer $n > 0$, we define $[n]=\{1,2,\cdots,n\}$. We use $\R_+$ and $\R_{\geq 0}$ to denote the set of positive and non-negative real numbers respectively. We use $0_n, 1_n \in \R^n$ to denote the all-zero and all-one vectors of size $n$.

Given a vector $v \in \R^n$, for any $m \leq n$, we use $v_{[:m]} \in \R^m$ to denote the first $m$ entries of $v$. For a vector $v \in \R^n$, we use $\diag(v) \in \R^{n \times n}$ to denote the diagonal matrix whose diagonal entries are $v$. For a square matrix $A \in \R^{n \times n}$, we use $\diag(A) \in \R^n$ to denote the vector of the diagonal entries of $A$. We use $\rk(A)$ to denote the rank of a matrix $A$. We use $\ker(A)$ and $\im(A)$ to denote the kernel space and the column space of $A$.

We say a matrix $A \in \R^{n \times n}$ is PSD (denoted as $A \succeq 0$) if $A$ is symmetric and $x^{\top} A x \geq 0$ for all $x \in \R^n$. We use $\mathbb{S}^{n \times n}$ to denote the set of PSD matrices of size $n \times n$. The spectral norm of a matrix $A \in \R^{n \times d}$ is defined as $\|A\|_2 = \max_{x\in \R^d, \|x\|_2 =1}\|Ax\|_2$.
The Frobenius norm of $A$ is defined as $\|A\|_{F} = \sqrt{\sum_{i \in [n]}\sum_{j \in [d]}A_{i, j}^2 }$. For any PSD matrix $M \in \mathbb{S}^{n \times n}$, we define the $M$-norm as $\|x\|_M = \sqrt{x^{\top} M x}$, $\forall x \in \R^n$.

We use $\Tmat(a,b,c)$ to denote the time to multiply two matrices of sizes $a \times b$ and $b \times c$. A basic fact of fast matrix multiplication is that $\Tmat(a,b,c) = \Tmat(b,c,a) = \Tmat(c,a,b)$ (see e.g. \cite{b13}), and we will use these three terms interchangeably.

\begin{fact}[Woodbury identity]\label{fac:woodbury}
Let $A \in \R^{n\times n}, C\in \R^{k\times k}, U\in \R^{n\times k}, V\in\R^{k\times n}$ where $A$ and $C$ are invertible, then
\begin{equation*}
(A + UCV)^{-1} = A^{-1} - A^{-1}U (C^{-1} + VA^{-1}U)^{-1}VA^{-1}.
\end{equation*}
\end{fact}

\begin{definition}[Hadamard product]
For any two matrices $A,B \in \R^{m \times n}$, the Hadamard product $A \circ B$ is defined as
\[
(A \circ B)_{i,j} = A_{i,j} \cdot B_{i,j}, ~~ \forall i \in [m], j \in [n].
\]
We also use $A^{\circ 2}$ to denote $A \circ A$.
\end{definition}

The Hadamard product has the following properties (the proofs are straightforward).
\begin{fact}[Properties of Hadamard product]\label{fac:Hadamard_product_property}
For matrices $A,B \in \R^{m \times n}$, and vectors $x \in \R^m$, $y \in \R^n$, we have the following properties.
\begin{enumerate}
\item \label{part:Hadamard_vector} $x^{\top} (A \circ B) y = \tr[\diag(x) A \diag(y) B^{\top}]$,
\item \label{part:Hadamard_rank_1} $A \circ (x \cdot y^{\top}) = \diag(x) \cdot A \cdot \diag(y)$.
\end{enumerate}
\end{fact}

\begin{definition}[Spectral approximation]\label{def:spectral_approx}
For any two symmetric matrices $A, \wt{A} \in \R^{n \times n}$, any parameter $\epsilon \in (0,1)$, we say $\wt{A}$ and $A$ are $\epsilon$-spectral approximation of each other, denoted as $\wt{A} \approx_{\epsilon} A$, if we have
\begin{equation*}
    e^{-\epsilon} \cdot x^{\top} A x \leq x^{\top} \wt{A} x \leq e^{\epsilon} \cdot x^{\top} A x, ~~ \forall x \in \R^n.
\end{equation*}
\end{definition}

Spectral approximation has the following properties (for completeness we include a proof in the appendix).
\begin{apxfactrep}[Properties of spectral approximation]\label{fac:spectral_approx}
For any two PSD matrices $A, \wt{A} \in \R^{n \times n}$, any parameter $\epsilon \in (0,1)$, if $\wt{A} \approx_{\epsilon} A$, then we have
\begin{enumerate}
\item \label{part:spectral_two_arms} $B^{\top} A B \approx_{\epsilon} B^{\top} \wt{A} B$, for any matrix $B \in \R^{n \times n}$.
\item \label{part:spectral_inverse} If both $A$ and $\wt{A}$ are invertible, then $A^{-1} \approx_{\epsilon} \wt{A}^{-1}$.
\item \label{part:PSD_trace} $e^{-\epsilon} \tr[A] \leq \tr[\wt{A}] \leq e^{\epsilon} \tr[A]$.
\item \label{part:spectral_hadamard} $\wt{A}^{\circ 2} \approx_{2 \epsilon} A^{\circ 2}$.
\end{enumerate}
\end{apxfactrep}
\begin{proof}
The proofs of the first three claims are straightforward. We only prove the last claim.

For any vector $x \in \R^n$, we have
\begin{equation*}
\begin{aligned}
x^{\top} A^{\circ 2} x = &~ \tr[\diag(x) A \diag(x) A] \\
= &~ \tr[A^{1/2} \diag(x) A \diag(x) A^{1/2}] \\
\leq &~ e^{\epsilon} \cdot \tr[A^{1/2} \diag(x) \wt{A} \diag(x) A^{1/2}] \\
= &~ e^{\epsilon} \cdot \tr[\wt{A}^{1/2} \diag(x) A \diag(x) \wt{A}^{1/2}] \\
\leq &~ e^{2\epsilon} \cdot \tr[\wt{A}^{1/2} \diag(x) \wt{A} \diag(x) \wt{A}^{1/2}] = e^{2\epsilon} \cdot x^{\top} \wt{A}^{\circ 2} x
\end{aligned}
\end{equation*}
where the first step follows from Fact~\ref{fac:Hadamard_product_property}, the second and the fourth steps follow from the trace invariance under cyclic permutations and the fact that $A^{1/2}$ exists when $A$ is PSD, the third and the fifth steps follow from Part~3 of this fact.

Similarly we can prove $x^{\top} A^{\circ 2} x \geq e^{-2\epsilon} \cdot x^{\top} \wt{A}^{\circ 2} x$. Thus we have $A^{\circ 2} \approx_{2 \epsilon} \wt{A}^{\circ 2}$.
\end{proof} 

\section{Background of sum-of-squares optimization} 
\label{sec_background}
In this section we provide the background of sum-of-squares optimization. We refer the readers to \cite{parrilo2020sum,py19} for more details.

\begin{definition}[Polynomial space]
We use $\V_{n,d}$ to denote the set of $n$-variate polynomials over the reals of degree at most $d$, where the degree means the total degree, i.e., the degree of $x_1^{d_1} \cdots x_n^{d_n}$ is $\sum_{i=1}^n d_i$.
\end{definition}

\begin{definition}[Degree of polynomial space]
   We define $L := \dim (\V_{n,d}) = \binom{n+d}{n}$ and $U := \dim (\V_{n,2d}) = \binom{n+2d}{n}$.
\end{definition}

After fixing a basis $(p_1, p_2, \cdots, p_L)$ of $\V_{n,d}$, there exists a one-to-one correspondence between any polynomial $p = \sum_{i=1}^L x_i \cdot p_i \in \V_{n,d}$ and the vector $[x_1, x_2, \cdots, x_L] \in \R^L$. From now on when the basis is clear from context, we will use $\V_{n,d}$ and $\R^L$ interchangeably, and similarly $\V_{n,2d}$ and $\R^U$ interchangeably. 

\begin{definition}[SOS polynomials]
    A polynomial $p \in \V_{n,2d}$ is said to be a sum-of-squares $(SOS)$ polynomial if $p$ can be written as a sum of squares of polynomials, i.e. $p = \sum_{i = 1}^M q_i^2$ for some $M \in \mathbb{N}$ and polynomials $q_1, q_2, \cdots, q_M \in \V_{n,d}$.
    
    We use $\Sigma_{n,2d}$ to denote the set of $n$-variate SOS polynomials of degree at most $2d$.
\end{definition}
The set $\Sigma_{n,2d}$ is a closed convex and pointed cone in $\V_{n,2d}$ with non-empty interior (Theorem 17.1 of \cite{nesterov2000squared}). The SOS optimization problem requires the variable $x \in \R^U$ to be in the SOS cone, and it is a special case of conic programming. Given a constraint matrix $A \in \R^{m \times U}$ where $m \leq U$, and $b \in \R^m$ and $c \in \R^U$, the SOS optimization can be written in the following primal-dual formulation:
\begin{equation}
  \label{SOS_program_primal_dual}
  \tag{SOS}
  \begin{aligned}
  \text{Primal:~~~} \min \; &\pr{c}{x} \quad \\
  \mathrm{s.t.}~~ Ax& =b \\
  x &\in \Sigma_{n,2d}\, ,\\
  \end{aligned}
  \quad\quad\quad
  \begin{aligned} 
  \text{Dual:~~~}\max \; & \pr{y}{b} \\
  \mathrm{s.t.}~~ A^\top y + s &= c \\
  s &\in \Sigma_{n,2d}^*\, . \\
  \end{aligned}
\end{equation}
Here $\Sigma_{n,2d}^* := \{s \in \R^U \mid s^{\top} x \geq 0, ~\forall x \in \Sigma_{n,2d}\}$ denotes the dual cone of $\Sigma_{n,2d}$.

Nesterov in \cite{nesterov2000squared} noted that the dual SOS cone allows the following characterization. 
\begin{theorem}[Dual cone characterization, Theorem 17.1 of \cite{nesterov2000squared}]\label{prop:nesterov_dual_SOS}
For any ordered bases $\mathbf{p} = (p_1, \ldots, p_L)$ and $\mathbf{q} = (q_1, \ldots, q_U)$ of $\mathcal{V}_{n,d}$ and $\mathcal{V}_{n,2d}$, let $\Lambda: \R^U \to \R^{L \times L}$ be the unique linear mapping satisfying $\Lambda(\mathbf{q}) = \mathbf{p}\mathbf{p}^\top$.\footnote{This equation means $\forall t \in \R^n$, $\Lambda([q_1(t), \cdots, q_U(t)]^\top) = [p_1(t), \cdots, p_L(t)]^\top \cdot [p_1(t), \cdots, p_L(t)]$.} Then the dual cone $\Sigma_{n,2d}^*$ admits the characterization under the bases $\mathbf{p}$ and $\mathbf{q}$:
  \begin{equation}
    \Sigma_{n,2d}^* = \big\{s \in \R^U \mid \Lambda(s) \succeq 0\big\}.
  \end{equation}
\end{theorem}

As barrier functions for the cone of positive semidefinite matrices are well-known, this also gives rise to a barrier function for the dual SOS cone. With the standard log-det barrier for the semidefinite cone, the following function $F: \Sigma_{n,2d}^* \to \R$ is a barrier function for $\Sigma_{n,2d}^*$:
\[
F(s) = - \log \det(\Lambda(s)).
\]
Furthermore, the barrier parameter $\nu_F$ of $F(s)$ is bounded by the barrier parameter $L$ of the original log-det barrier function (\cite{nesterov2000squared}).

\paragraph*{Interpolant basis}
The barrier function depends on the choice of the basis for both $\mathcal V_{n,d}$ and $\mathcal V_{n,2d}$, as the linear map $\Lambda$ depends on these two bases. We follow the approach of \cite{py19} and focus on the so-called \emph{interpolant} bases, which generalises well to multivariate polynomials and is numerically stable. 

For the vector space $\mathcal{V}_{n,2d}$, consider a set of \emph{unisolvent} points $\mathcal{T} = \{t_1, t_2, \cdots, t_U\} \subseteq \R^n$, which is a set points such that every polynomial in $\mathcal{V}_{n,2d}$ is uniquely determined by its values on the points in $\mathcal T$.
For univariate polynomials any set of $U$ points suffices, but this does not hold anymore for the multivariate case. To also ensure numerical stability, the so called (approximate) Fekete points can be used as unisolvent points \cite{sommariva09, sommariva10}.

The interpolant basis is defined as follows. Let us fix a set of unisolvent points $\mathcal{T} = \{t_1, t_2, \cdots, t_U\} \subseteq \R^n$. Now every $t_u \in \mathcal T$ implies a Lagrange polynomial $q_u$ which is the unique polynomial that satisfies $q_u(t_u) = 1$ and $q_u(t_v) = 0$ for all $t_v \neq t_u \in \mathcal T$. The Lagrange polynomials form a basis $\mathbf{q} = (q_1, \cdots, q_U)$ of $\V_{n,2d}$. Choose any basis $\mathbf{p} = (p_1, \ldots, p_L)$ of $\V_{n,d}$. Define the matrix $P \in \R^{U \times L}$ as 
\[
P_{u, \ell} = p_\ell(t_u), ~~ \forall u \in [U], \ell \in [L].
\] 
By the definition of the Lagrange polynomials, $p_i p_j = \sum_{u=1}^U p_i(t_u) p_j(t_u) q_u$, so we have $\mathbf{p} \mathbf{p}^{\top} = P^{\top} \diag(\mathbf{q}) P$. Thus under the bases $\mathbf{p}$ and $\mathbf{q}$, the linear map $\Lambda: \R^U \to \mathbb{R}^{L \times L}$ takes on the following convenient form:
\begin{equation}\label{eq:def_Lambda}
    \Lambda(s) = P^{\top} \diag(s) P.
\end{equation}

 \section{Algorithm}
Since in this paper we focus on the theoretical running time of the algorithm, for simplicity we use the barrier method (see e.g.\  \cite[Chapter~2]{renegarMathematicalViewInteriorPoint2001}) instead of the more sophisticated Skajaa–Ye Algorithm used by \cite{py19}.

The dual formulation of \eqref{SOS_program_primal_dual} is equivalent to the following optimization problem
\begin{equation*}
    \min -b^{\top} y ~~~
    s.t.~~ y \in \ov{D}_F,
\end{equation*}
where with an abuse of the notation we define $F: \R^m \to \R_+$ to be the barrier function
\begin{equation}\label{eq:barrier}
    F(y) = -\log \det(\Lambda(c - A^{\top} y))
\end{equation}
for $c \in \R^U$, and $A \in \R^{m \times U}$, and $\Lambda(s) = P^{\top} \diag(s) P$ is the linear operator defined in Eq~\eqref{eq:def_Lambda}.
$D_F \subseteq \R^m$ is the domain of $F$, and $\ov{D}_F$ is the closure of $D_F$.

The barrier parameter of the barrier function $F$ is $\nu_F = L$. The gradient and the Hessian of the barrier function $F$ are (define $s := c - A^{\top} y$):
\begin{equation*}
\begin{aligned}
    g(y) = &~ A \cdot \diag\Big( P \big( P^{\top} \diag(s) P \big)^{-1} P^{\top} \Big), \\
    H(y) = &~ A \cdot \Big( P \big( P^{\top} \diag(s) P \big)^{-1} P^{\top} \Big)^{\circ 2} \cdot A^{\top}.
\end{aligned}
\end{equation*}

For any $\eta > 0$, define a function $F_{\eta}: \R^m \to \R$:
\begin{equation*}
    F_{\eta}(y) = -\eta \cdot b^{\top} y + F(y).
\end{equation*}
The gradient and the Hessian of $F_{\eta}(y)$ are:
\begin{equation*}
\begin{aligned}
    g_{\eta}(y) = &~ -\eta \cdot b + A \cdot \diag\Big( P \big( P^{\top} \diag(s) P \big)^{-1} P^{\top} \Big), \\
    H_{\eta}(y) = &~ A \cdot \Big( P \big( P^{\top} \diag(s) P \big)^{-1} P^{\top} \Big)^{\circ 2} \cdot A^{\top}.
\end{aligned}
\end{equation*}
Note that $H_{\eta}(y) = H(y)$ for any $\eta$.

In each iteration the barrier method increases $\eta$ by a factor of $1 + \Theta( \frac{1}{\sqrt{L}})$, and it performs a Newton step
\[
y \gets y - H_{\eta}(y)^{-1} \cdot g_{\eta}(y).
\]
By standard IPM theory it suffices to use a spectral approximation of the Hessian matrix in the Newton step. For more details see e.g. \cite{renegarMathematicalViewInteriorPoint2001}.

The main technical part of our algorithm is to efficiently maintain a matrix $N$ that is the spectral approximation of the inverse of the Hessian matrix. To do this, we maintain another matrix $\wt{S}$ that is a spectral approximation of $S := P^{\top} \diag(s) P$, and we use the subroutine \textsc{LowRankUpdate}(Algorithm~\ref{alg:low_rank_update}, Lemma~\ref{lem:low_rank_update}) to update $\wt{S}$. After $\wt{S}$ is updated, we use another subroutine \textsc{UpdateHessianInv} (Algorithm~\ref{alg:hessian_inverse_update}, Lemma~\ref{lem:hessian_inverse_update}) to update $N$. A complete description of our algorithm can be found in Algorithm~\ref{alg:barrier}.

\begin{algorithm}[th!]
    \caption{Main SOS algorithm.}
    \label{alg:barrier}
    \SetKwInOut{Input}{Input}
    \SetKwInOut{Parameters}{Parameters}
    \SetKwInOut{Output}{Output}
    \SetKw{And}{\textbf{and}}
    \Parameters{$\delta \in (0, 1)$, $\epsilon_N \in (0, 0.05)$, $\alpha = \frac{\epsilon_N}{20 \sqrt{L}}$, $t = 40 \epsilon_N^{-1} \sqrt{L} \log(L/\delta)$.}
    \Input{$A \in \R^{m \times U}$, $b \in \R^m$, $c \in \R^U$}
    \Output{A near feasible and optimal solution.}

    Construct $P \in \R^{U \times L}$ of the interpolant basis. Convert $A, b, c$ to the interpolant basis.
    
    Use Lemma~\ref{lem:init} to obtain a modified dual SOS optimization problem which has an initial solution $(y, s) \in \R^{m} \times \R^{U}$ that is optimal for $F_{\eta}$, where $\eta = 1$.
    
    $\wt{S} \leftarrow S \leftarrow P^{\top} \diag(s) P$ \tcp*{$\wt{S}, S \in \R^{L \times L}$}\label{line:initial_S}
    
    $T \leftarrow S^{-1}$ \tcp*{$T \in \R^{L \times L}$}
    
    $N \leftarrow \big( A (P T P^{\top})^{\circ 2} A^{\top} \big)^{-1}$ \tcp*{$N \in \R^{m \times m}$} \label{line:N_init}

    $g \leftarrow -\eta \cdot b + A \cdot \diag\big( P ( P^{\top} \diag(s) P )^{-1} P^{\top} \big)$ \tcp*{$g \in \R^m$} \label{line:g_init}
    
    \For{$i = 1, 2, \cdots, t$}{
        $\delta_y \leftarrow - N \cdot g$ \tcp*{$\delta_y \in \R^m$}\label{line:delta_y}
        
        $y^{\new} \leftarrow y + \delta_y$ \tcp*{$y^{\new} \in \R^m$}\label{line:y_new}
        
        $s^{\new} \leftarrow c - A^{\top} y^{\new}$ \tcp*{$s^{\new} \in \R^U$}
        
        $\eta^{\new} \leftarrow \eta \cdot (1 + \alpha)$\;\label{line:eta}
        
        $S^{\new} \leftarrow P^{\top} \diag(s^{\new}) P$ \tcp*{$S^{\new} \in \R^{L \times L}$} \label{line:compute_S_new}
        
        $\wt{S}^{\new}, V_1, V_2 \leftarrow \textsc{LowRankUpdate}(S^{\new}, \wt{S})$\;
        
        \hfill \tcp{Lemma~\ref{lem:low_rank_update}, $\wt{S}^{\new} \in \R^{L \times L}$, $V_1, V_2 \in \R^{L \times r_i}$ or $V_1 = V_2 = \mathtt{null}$}\label{line:low_rank_update}
        
        \If{$V_1 = V_2 = \mathtt{null}$}{\label{line:if_start}
            
            $T^{\new} \leftarrow (\wt{S}^{\new})^{-1}$ \tcp*{$T^{\new} \in \R^{L \times L}$}
            
            $N^{\new} \leftarrow \big( A \cdot (P T^{\new} P^{\top})^{\circ 2} \cdot A^{\top} \big)^{-1}$ \tcp*{$N^{\new} \in \R^{m \times m}$}
        }
        \Else{
        
        $T^{\new}, N^{\new} \leftarrow \textsc{UpdateHessianInv}(T,N,V_1,V_2)$\; \label{line:hessian_inverse_update} 
        
        \hfill \tcp{Lemma~\ref{lem:hessian_inverse_update}, $T^{\new} \in \R^{L \times L}, N^{\new} \in \R^{m \times m}$}
        }\label{line:if_end}
        
        $g^{\new} \leftarrow -\eta^{\new} \cdot b + A \cdot \diag\Big( P \big( P^{\top} \diag(s^{\new}) P \big)^{-1} P^{\top} \Big)$ \tcp*{$g^{\new} \in \R^m$}\label{line:g}
        
        $(\eta, y, s, \wt{S}, T, N, g) \leftarrow (\eta^{\new}, y^{\new}, s^{\new}, \wt{S}^{\new}, T^{\new}, N^{\new}, g^{\new})$\;
    }
  \Return{$(y,s)$}
\end{algorithm}
 \section{Updating Hessian inverse efficiently}\label{sec:hessian_inverse}
In this section we prove how to update the Hessian inverse efficiently. We present the algorithm \textsc{UpdateHessianInv} in Algorithm~\ref{alg:hessian_inverse_update}.
\begin{algorithm}[!ht]
    \caption{\textsc{UpdateHessianInv}}
    \label{alg:hessian_inverse_update}
    \SetKwInOut{Input}{Input}
    \SetKwInOut{Output}{Output}
    \SetKw{And}{\textbf{and}}
    \Input{$T \in \R^{L \times L}$, $N \in \R^{m \times m}$, $V_1,V_2 \in \R^{L \times r}$} 
    \Output{$T^{\new} \in \R^{L \times L}$, $N^{\new} \in \R^{m \times m}$}
    
    \tcp{Step 1}
    
    $\ov{V_1} \leftarrow - T V_1 \cdot (I + V_2^{\top} T V_1 )^{-1}$ \tcp*{$\ov{V_1} \in \R^{L \times r}$}
    
    $\ov{V_2} \leftarrow T V_2$ \tcp*{$\ov{V_2} \in \R^{L \times r}$}

    $T^{\new} \leftarrow T + \ov{V_1} \cdot \ov{V_2}^{\top}$ \tcp*{$T^{\new} \in \R^{L \times L}$}
    
    \tcp{Step 2}
    
    $Y' \leftarrow [2 P T, P \ov{V_1}]$ \tcp*{$Y' \in \R^{U \times (L+r)}$}
    
    $Z' \leftarrow [P, P \ov{V_2}]$ \tcp*{$Z' \in \R^{U \times (L+r)}$}
    
    $Y \leftarrow [\diag(u_1)Y', \cdots, \diag(u_r)Y']$, $u_i$ is the $i$-th column of $P \ov{V_1}$ \tcp*{$Y \in \R^{U \times (L+r) r}$}
    
    $Z \leftarrow [\diag(v_1)Z', \cdots, \diag(v_r)Z']$, $v_i$ is the $i$-th column of $P \ov{V_2}$ \tcp*{$Z \in \R^{U \times (L+r) r}$}
    
    \tcp{Step 3}
    
    $N^{\new} \leftarrow N - N \cdot (AY) \cdot \big(I + (AZ)^{\top} N (AY) \big)^{-1} \cdot (AZ)^{\top} \cdot N$ \tcp*{$N^{\new} \in \R^{m \times m}$}\label{line:N_new}
    
    \Return{$T^{\new}, N^{\new}$}
\end{algorithm}

\begin{lemma}[Hessian inverse update]\label{lem:hessian_inverse_update}
In the algorithm $\textsc{UpdateHessianInv}$ (Algorithm~\ref{alg:hessian_inverse_update}), the inputs are the maintained matrices $T, N$ and the updates $V_1, V_2 \in \R^{L \times r}$ where $r$ satisfies $Lr \leq U$. The inputs satisfy that for some $\wt{S} \in \mathbb{S}^{L \times L}$,
\begin{equation*}
\begin{aligned}
    & T = \wt{S}^{-1} \in \R^{L \times L}, \\
    & N = \big( A \cdot (P \wt{S}^{-1} P^{\top})^{\circ 2} \cdot A^{\top} \big)^{-1} \in \R^{m \times m},
\end{aligned}
\end{equation*}
Let $\wt{S}^{\new} = \wt{S} + V_1 V_2^{\top}$. The algorithm outputs two matrices $T^{\new}$, $N^{\new}$ such that
\begin{equation*}
\begin{aligned}
    & T^{\new} = (\wt{S}^{\new})^{-1} \in \R^{L \times L}, \\
    & N^{\new} = \big( A \cdot (P (\wt{S}^{\new})^{-1} P^{\top})^{\circ 2} \cdot A^{\top} \big)^{-1} \in \R^{m \times m}.
\end{aligned}
\end{equation*}
Furthermore, the algorithm takes $O(\Tmat(U, U, Lr))$ time.
\end{lemma}
\begin{proof}
We first prove the correctness by analyzing each step of the algorithm.

\noindent {\bf Step 1. \ \  Compute $\ov{V_1}, \ov{V_2} \in \R^{L \times r}$ and $T^{\new} \in \R^{L \times L}$.}
\begin{equation*}
\begin{aligned}
    T^{\new} = &~ T + \ov{V_1} \cdot \ov{V_2}^{\top} \\     
    = &~ T - T V_1 \cdot (I + V_2^{\top} T V_1 )^{-1} \cdot V_2^{\top} T^{\top} \\
    = &~ (\wt{S} + V_1 V_2^{\top})^{-1} = (\wt{S}^{\new})^{-1},
\end{aligned}
\end{equation*}
where the first two steps follow from algorithm description, the third step follows from the Woodbury identity (Fact~\ref{fac:woodbury}) and $T = \wt{S}^{-1}$.

Thus $T^{\new}$ satisfies the requirement of the output.

\noindent {\bf Step 2. \ \  Compute $Y', Z' \in \R^{U \times (L+r)}$ and $Y, Z \in \R^{U \times (L+r) r}$}

We prove that $Y$ and $Z$ satisfy $(P T P^{\top})^{\circ 2} + Y \cdot Z^{\top} = (P T^{\new} P^{\top})^{\circ 2}$:
\begin{equation*}
\begin{aligned}
    (P T P^{\top})^{\circ 2} + Y \cdot Z^{\top}
    = &~ (P T P^{\top})^{\circ 2} + \sum_{i=1}^r \diag(u_i) \cdot \big( Y' \cdot (Z')^{\top} \big) \cdot \diag(v_i) \\
    = &~ (P T P^{\top})^{\circ 2} + \big( Y' \cdot (Z')^{\top} \big) \circ \big( \sum_{i=1}^r u_i \cdot v_i^{\top} \big) \\
    = &~ (P T P^{\top})^{\circ 2} + \big( 2 P T P^{\top} + (P \ov{V_1}) \cdot (P \ov{V_2})^{\top} \big) \circ \big( (P \ov{V_1})\cdot (P \ov{V_2})^\top \big) \\
    = &~ \big(P T P^{\top} + (P \ov{V_1}) \cdot (P \ov{V_2})^{\top} \big)^{\circ 2} \\
    = &~ (P T^{\new} P^{\top})^{\circ 2},
\end{aligned}
\end{equation*}
where the first step follows from the algorithm description of $Y$ and $Z$, the second step follows from Part~\ref{part:Hadamard_rank_1} of Fact~\ref{fac:Hadamard_product_property} that $\diag(x) \cdot A \cdot \diag(y) = A \circ (x \cdot y^{\top})$, the third step follows from $Y' \cdot (Z')^{\top} = 2 P T P^{\top} + (P \ov{V_1}) \cdot (P \ov{V_2})$ and $(P \ov{V_1}) \cdot (P \ov{V_2}) = \sum_{i=1}^r u_i \cdot v_i^{\top}$ (see algorithm description of $Y'$ and $Z'$), the last step follows from $T^{\new} = T + \ov{V_1} \cdot \ov{V_2}^{\top}$.

\noindent {\bf Step 3. \ \  Compute $N^{\new} \in \R^{m \times m}$.}
\begin{equation*}
\begin{aligned}
    N^{\new}
    = &~ N - N \cdot (AY) \cdot \big(I + (AZ)^{\top} N (AY) \big)^{-1} \cdot (AZ)^{\top} \cdot N \\
    = &~ \big( A \cdot (P T P^{\top})^{\circ 2} \cdot A^{\top} + (A Y) \cdot (A Z)^{\top} \big)^{-1} \\
    = &~ \big( A \cdot (P T^{\new} P^{\top})^{\circ 2} \cdot A^{\top} \big)^{-1},
\end{aligned}
\end{equation*}
where the first step follows from the algorithm description of $N^{\new}$, the second step follows from $N = \big( A \cdot (P T P^{\top})^{\circ 2} \cdot A^{\top} \big)^{-1}$ and the Woodbury identity (Fact~\ref{fac:woodbury}), and the last step follows from $(P T^{\new} P^{\top})^{\circ 2} = (P T P^{\top})^{\circ 2} + Y \cdot Z^{\top}$. 

Thus $N^{\new}$ satisfies the requirement of the output.

\noindent {\bf Time complexity.} It is easy to see that the most time-consuming step is to compute $N^{\new}$ on Line~\ref{line:N_new}, and in total this step takes $O(\Tmat(m,U,Lr) + \Tmat(m,m,Lr) + (Lr)^{\omega})$ time.

Since $Lr \leq U$ and $m \leq U$, overall this algorithm takes at most $O(\Tmat(U,U,Lr))$ time.
\end{proof}

\newpage
\section{Correctness}

\subsection{Standard results from IPM theory}\label{sec:standard}
We use the following two results of the barrier method that hold for any cone with a barrier function. The proofs are standard, (see e.g., \cite[Section~2.4]{renegarMathematicalViewInteriorPoint2001}), and for completeness we include a proof in the Appendix.
\begin{apxlemmarep}[Invariance of Newton step, \cite{renegarMathematicalViewInteriorPoint2001}]\label{lem:invariant_newton}
Consider the following optimization problem: $\min -b^{\top} y$ s.t. $y \in \ov{D}_F$, where $F: \R^m \to \R_+$ is a barrier function with barrier parameter $\nu_F$, $D_F \subseteq \R^m$ is the domain of $F$, and $\ov{D}_F$ is the closure of $D_F$. For any $\eta \geq 1$, define $F_{\eta}(y) = -\eta b^{\top} y + F(y)$. Let $g_{\eta}(y) \in \R^m$ and $H(y) \in \R^{m \times m}$ denote the gradient and the Hessian of $F_{\eta}$ at $y$.

Let $0 < \epsilon_N \leq 0.05$ be a parameter. If a feasible solution $y \in D_F$, a parameter $\eta > 0$, and a positive definite matrix $\wt{H} \in \mathbb{S}^{n \times n}$ satisfy the following: 
\begin{equation*}
\|g_{\eta}(y)\|_{H(y)^{-1}} \leq \epsilon_N,~~~
\wt{H} \approx_{0.02} H(y).
\end{equation*}
Then $\eta^{\new} = \eta \cdot (1 + \frac{\epsilon_N}{20 \sqrt{\nu_F}} )$, $y^{\new} = y + \delta_y$ where $\delta_y = - \wt{H}^{-1} g_{\eta^{\new}}(y)$ satisfy $y^{\new} \in D_F$ and
\begin{equation*}
\|\delta_y\|_{H(y)} \leq 2 \epsilon_N, ~~~
\|g_{\eta^{\new}} (y^{\new})\|_{H(y^{\new})^{-1}} \leq \epsilon_N.
\end{equation*}
\end{apxlemmarep}
\begin{proof}
For clarity in the proof we use $\eta_1 := \eta$, $\eta_2 := \eta^{\new}$, $y_1 := y$, and $\wt{y}_2 := y^{\new} = y_1 - \wt{H}^{-1} g_{\eta_2}(y_1)$. We also define $\alpha_H := e^{0.02} \leq 1.03$. Note that $\alpha_H^{-1} H(y_1) \preceq \wt{H} \preceq \alpha_H H(y_1)$. 

We first introduce some notation:
\begin{equation*}
\begin{aligned}
    n_{\eta_1}(y_1) := &~ - H(y_1)^{-1} g_{\eta_1}(y_1), \\
    n_{\eta_2}(y_1) := &~ - H(y_1)^{-1} g_{\eta_2}(y_1), \\
    \wt{n}_{\eta_2}(y_1) := &~ - \wt{H}(y_1)^{-1} g_{\eta_2}(y_1), \\
    n_{\eta_2}(\wt{y}_2) := &~ -H(\wt{y}_2)^{-1} g_{\eta_2}(\wt{y}_2)
\end{aligned}
\end{equation*}
Note that $\wt{y}_2 = y_1 + \wt{n}_{\eta_2}(y_1)$. Recall that we define $\|x\|_{H(y)}^2 = x^{\top} H(y) x, \forall x \in \R^m$ to be the local norm at $y$. The local norm induces a matrix norm $\|M\|_{H(y)} = \max_{x} \frac{\|M x\|_{H(y)}}{\|x\|_{H(y)}} = \|H(Y)^{1/2} M H(Y)^{-1/2}\|_2, \forall M \in \R^{m \times m}$.

{\bf Step 1 (Bound $\|n_{\eta_2}(y_1)\|_{H(y_1)}$).}
We have
\begin{equation*}
\begin{aligned}
    n_{\eta_2}(y_1) = &~ - H(y_1)^{-1} g_{\eta_2}(y_1) \\
    = &~ - H(y_1)^{-1} \cdot (-\eta_2 \cdot b + g(y_1)) \\
    = &~ \frac{\eta_2}{\eta_1} \cdot n_{\eta_1}(y_1) + (\frac{\eta_2}{\eta_1} - 1) \cdot H(y_1)^{-1} g(y_1).
\end{aligned}
\end{equation*}
Thus we have
\begin{equation}\label{eq:inv_newton_1}
\begin{aligned}
    \|n_{\eta_2}(y_1)\|_{H(y_1)} \leq &~ \frac{\eta_2}{\eta_1} \cdot \|n_{\eta_1}(y_1)\|_{H(y_1)} + |\frac{\eta_2}{\eta_1} - 1| \cdot \|H(y_1)^{-1} g(y_1)\|_{H(y_1)} \\
    \leq &~ \frac{\eta_2}{\eta_1} \cdot \|n_{\eta_1}(y_1)\|_{H(y_1)} + |\frac{\eta_2}{\eta_1} - 1| \cdot \sqrt{\nu_F} \\
    \leq &~ 1.1 \epsilon_N,
\end{aligned}
\end{equation}
where the second step follows from the definition of the barrier parameter (Section~2.3.1 of \cite{renegarMathematicalViewInteriorPoint2001}), and the third step follows from $\eta_2 = \eta_1 \cdot (1 + \frac{\epsilon_N}{20 \sqrt{\nu_F}})$ and $\|n_{\eta_1}(y_1)\|_{H(y_1)} \leq \epsilon_N$.

{\bf Step 2 (Bound $\|\wt{n}_{\eta_2}(y_1)\|_{H(y_1)}$).}
We have
\begin{equation}\label{eq:inv_newton_2}
\begin{aligned}
    \|\wt{n}_{\eta_2}(y_1)\|_{H(y_1)}^2 = &~ g_{\eta_2}(y_1)^{\top} \wt{H}^{-1} H(y_1) \wt{H}^{-1} g_{\eta_2}(y_1) \\
    \leq &~ \alpha_H \cdot g_{\eta_2}(y_1)^{\top} \wt{H}^{-1} g_{\eta_2}(y_1) \\
    \leq &~ \alpha_H^2 \cdot g_{\eta_2}(y_1)^{\top} H(y_1)^{-1} g_{\eta_2}(y_1) \\
    = &~ \alpha_H^2 \cdot \|n_{\eta_2}(y_1)\|_{H(y_1)}^2 \\
    \leq &~ 4 \epsilon_N^2,
\end{aligned}
\end{equation}
where the second step follows from $H(y_1) \preceq \alpha_H \wt{H}$, the third step follows from $\wt{H}^{-1} \preceq \alpha_H H(y_1)^{-1}$, and the fifth step follows from $\alpha_H \leq 1.03$ and $\|n_{\eta_2}(y_1)\|_{H(y_1)} \leq 1.1 \epsilon_N$ (Eq.~\eqref{eq:inv_newton_1}).

Note that $\wt{y}_2 \in D_F$ directly follows from this bound and the definition of self-concordant functions. This proves the first inequality that of $\|\delta_y\|_{H(y)} \leq 2 \epsilon_N$.

{\bf Step 3 (Bound $\|H(\wt{y}_2)^{-1} H(y_1)\|_{H(y_1)}$).}
We have
\begin{equation}\label{eq:inv_newton_3}
\begin{aligned}
    \|H(\wt{y}_2)^{-1} H(y_1)\|_{H(y_1)} \leq &~ \frac{1}{(1 - \|\wt{y}_2 - y_1\|_{H(y_1)})^2} \\
    = &~ \frac{1}{(1 - \|\wt{n}_{\eta_2}(y_1)\|_{H(y_1)})^2} \\
    \leq &~ 1/(0.9)^2
\end{aligned}
\end{equation}
where the first step follows from Theorem~2.2.1 of \cite{renegarMathematicalViewInteriorPoint2001}, the second step follows from $\wt{y}_2 = y_1 + \wt{n}_{\eta_2}(y_1)$, the third step follows from $\|\wt{n}_{\eta_2}(y_1)\|_{H(y_1)} \leq 2 \epsilon_N$ (Eq.~\eqref{eq:inv_newton_2}) and $\epsilon_N \leq 0.05$.

{\bf Step 4 (Bound $\|H(y_1)^{-1} g_{\eta_2}(\wt{y}_2)\|_{H(y_1)}$).}
We have
\begin{equation}\label{eq:inv_newton_4_1}
\begin{aligned}
    \|H(y_1)^{-1} g_{\eta_2}(\wt{y}_2)\|_{H(y_1)}
    \leq &~ \underbrace{\|H(y_1)^{-1} \cdot \big( g_{\eta_2}(\wt{y}_2) - g_{\eta_2}(y_1) - H(y_1) \wt{n}_{\eta_2}(y_1) \big)\|_{H(y_1)}}_{a_1} \\
    &~ + \underbrace{\|H(y_1)^{-1} g_{\eta_2}(y_1) + \wt{n}_{\eta_2}(y_1)\|_{H(y_1)}}_{a_2}.
\end{aligned}
\end{equation}
Next we bound these two terms separately. For the first term, we have
\begin{equation}\label{eq:inv_newton_4_2}
\begin{aligned}
    a_1 = &~ \|H(y_1)^{-1} \cdot \int_0^1 [H(y_1 + t \wt{n}_{\eta_2}(y_1)) - H(y_1)] \cdot \wt{n}_{\eta_2}(y_1) ~\mathrm{d}t \|_{H(y_1)} \\
    \leq &~ \|\wt{n}_{\eta_2}(y_1)\|_{H(y_1)} \cdot \int_0^1 \|I - H(y_1)^{-1} \cdot H(y_1 + t \wt{n}_{\eta_2}(y_1))\|_{H(y_1)} ~\mathrm{d}t \\
    \leq &~ \|\wt{n}_{\eta_2}(y_1)\|_{H(y_1)} \cdot \int_0^1 \Big( \frac{1}{(1 - t \|\wt{n}_{\eta_2}(y_1)\|_{H(y_1)})^2} - 1\Big) \mathrm{d}t \\
    = &~ \frac{\|\wt{n}_{\eta_2}(y_1)\|_{H(y_1)}^2}{1 - \|\wt{n}_{\eta_2}(y_1)\|_{H(y_1)}} \\
    \leq &~ 5 \epsilon_N^2,
\end{aligned}
\end{equation}
where the first step follows from $\wt{y}_2 - y_1 = \wt{n}_{\eta_2}(y_1)$ and Proposition~1.5.7 of \cite{renegarMathematicalViewInteriorPoint2001}, the third step follows from Theorem~2.2.1 of \cite{renegarMathematicalViewInteriorPoint2001}, and the fifth step follows from $\|\wt{n}_{\eta_2}(y_1)\|_{H(y_1)} \leq 2 \epsilon_N$ (Eq.~\eqref{eq:inv_newton_2}) and $\epsilon_N \leq 0.05$.

For the second term, we have
\begin{equation}\label{eq:inv_newton_4_3}
\begin{aligned}
    a_2^2 = &~ \Big| g_{\eta_2}(y_1)^{\top} \cdot \Big( H(y_1)^{-1} - \wt{H}^{-1} H(y_1) \wt{H}^{-1} \Big) \cdot g_{\eta_2}(y_1) \Big| \\
    \leq &~ \max\{\alpha_H^2 - 1, 1 - \alpha_H^{-2}\} \cdot g_{\eta_2}(y_1)^{\top} \cdot H(y_1)^{-1} \cdot g_{\eta_2}(y_1) \\
    \leq &~ 0.09 \epsilon_N^2
\end{aligned}
\end{equation}
where the first step follows from $\wt{n}_{\eta_2}(y_1) = - \wt{H}^{-1} \cdot g_{\eta_2}(y_1)$, the second step follows from $\alpha_H^{-1} H(y_1) \preceq \wt{H} \preceq \alpha_H H(y_1)$, the third step follows from $\|n_{\eta_2}(y_1)\|_{H(y_1)} \leq 1.1 \epsilon_N$ (Eq.~\eqref{eq:inv_newton_1}) and $\alpha_N \leq 1.03$.

Plug in Eq.~\eqref{eq:inv_newton_4_2} and \eqref{eq:inv_newton_4_3} into Eq.~\eqref{eq:inv_newton_4_1}, we have
\begin{equation}\label{eq:inv_newton_4}
    \|H(y_1)^{-1} g_{\eta_2}(\wt{y}_2)\|_{H(y_1)}
    \leq 5 \epsilon_N^2 + 0.3 \epsilon_N \leq 0.6 \epsilon_N,
\end{equation}
where the last step follows from $\epsilon_N \leq 0.05$.

{\bf Step 5 (Bound $\|n_{\eta_2}(\wt{y}_2)\|_{H(\wt{y}_2)}$).}
We are finally ready to bound $\|n_{\eta_2}(\wt{y}_2)\|_{H(\wt{y}_2)}$. We have
\begin{equation*}
\begin{aligned}
    \|n_{\eta_2}(\wt{y}_2)\|_{H(\wt{y}_2)}^2 = &~ g_{\eta_2}(\wt{y}_2)^{\top} H(\wt{y}_2)^{-1} g_{\eta_2}(\wt{y}_2) \\
    \leq &~ \|H(\wt{y}_2)^{-1} H(y_1)\|_{H(y_1)}^2 \cdot \|H(y_1)^{-1} g_{\eta_2}(\wt{y}_2)\|_{H(y_1)}^2 \\
    \leq &~ 1/(0.9)^4 \cdot (0.6 \epsilon_N)^2 \leq \epsilon_N^2,
\end{aligned}
\end{equation*}
where the third step follows from Eq.~\eqref{eq:inv_newton_3} and Eq.~\eqref{eq:inv_newton_4}. This proves the second inequality that $\|g_{\eta^{\new}} (y^{\new})\|_{H(y^{\new})^{-1}} \leq \epsilon_N$.
\end{proof}

\begin{apxlemmarep}[Approximate optimality, \cite{renegarMathematicalViewInteriorPoint2001}]\label{lem:approximate_optimality}
Consider the following optimization problem: $\min -b^{\top} y$ s.t. $y \in \ov{D}_F$, where $F: \R^m \to \R_+$ is a barrier function with barrier parameter $\nu_F$, $D_F \subseteq \R^m$ is the domain of $F$, and $\ov{D}_F$ is the closure of $D_F$. Let $\mathrm{OPT}$ be the optimal objective value of this optimization problem. For any $\eta \geq 1$, define $F_{\eta}(y) = -\eta b^{\top} y + F(y)$. Let $g_{\eta}(y) \in \R^m$ and $H(y) \in \R^{m \times m}$ denote the gradient and the Hessian of $F_{\eta}$ at $y$.

Let $0 < \epsilon_N \leq 0.05$. If a feasible solution $y \in D_F$ satisfies $\|g_\eta(y)\|_{H(y)^{-1}} \leq \epsilon_N$,
then we have $-b^\top y \leq \mathrm{OPT} + \frac{\nu_F}{\eta} \cdot (1 + 2 \epsilon_N)$.
\end{apxlemmarep}
\begin{proof}
We use $g(y) \in \R^m$ to denote the gradient of $F$. Note that the Hessian of $F$ and $F_{\eta}$ are the same. Let $z(\eta) \in D_F$ be the minimizer of $F_{\eta}(y) = -\eta \langle b, y \rangle + F(y)$. Let $y^* \in \ov{D}_F$ denote the optimal solution of $\min - \langle b, y \rangle$ s.t. $y \in \ov{D}_F$.

{\bf Step 1.} We have
\begin{equation}\label{eq:approx_opt_1}
\begin{aligned}
    -\langle b, z(\eta) \rangle - \mathrm{OPT} = &~ -\langle b, z(\eta) \rangle + \langle b, y^* \rangle \\
    = &~ \frac{1}{\eta} \langle \eta b, y^* - z(\eta) \rangle \\
    = &~ \frac{1}{\eta} \langle g(z(\eta)), y^* - z(\eta) \rangle < \frac{1}{\eta} \nu_F,
\end{aligned}
\end{equation}
where the second step follows from $g_{\eta} (z(\eta)) = -\eta b + g(z(\eta)) = 0$ since $z(\eta)$ is the minimizer of $F_{\eta}$, the third step follows from Theorem 2.3.3 of \cite{renegarMathematicalViewInteriorPoint2001} that $\langle g(x), y - x \rangle < \nu_F$ for any $x, y \in D_F$.

{\bf Step 2.} By the definition of self-concordant functions (Section 2.2.1 of \cite{renegarMathematicalViewInteriorPoint2001}), for any $z \in D_F$, we have that $z + t\cdot H(z)^{-1} b \in D_F$ if $t \in [0, \frac{1}{\|H(z)^{-1} b\|_{H(z)}})$. So we have
\begin{equation*}
    -\langle b, z + t\cdot H(z)^{-1} b \rangle \geq \mathrm{OPT},
\end{equation*}
which implies that
\begin{equation*}
    \langle b, t\cdot H(z)^{-1} b \rangle \leq -\langle b, z \rangle - \mathrm{OPT},
\end{equation*}
and as $t$ approaches $\frac{1}{\|H(z)^{-1} b\|_{H(z)}} = \frac{1}{\|b\|_{H(z)^{-1}}}$, we have
\begin{equation}\label{eq:approx_opt_2}
    \|b\|_{H(z)^{-1}} \leq -\langle b, z \rangle - \mathrm{OPT}.
\end{equation}

{\bf Step 3.} Hence for all $y, z \in D_F$, we have
\begin{equation}\label{eq:approx_opt_3}
\begin{aligned}
    \frac{-\langle b, y \rangle - \mathrm{OPT}}{-\langle b, z \rangle - \mathrm{OPT}} = &~ 1 + \frac{-\langle b, y - z \rangle}{-\langle b, z \rangle - \mathrm{OPT}} \\
    \leq &~ 1 + \frac{\|b\|_{H(z)^{-1}} \cdot \|y - z\|_{H(z)}}{-\langle b, z \rangle - \mathrm{OPT}} \\
    \leq &~ 1 + \|y - z\|_{H(z)},
\end{aligned}
\end{equation}
where the last step follows from Eq.~\eqref{eq:approx_opt_2}.

{\bf Step 4.}
Plugging in $z = z(\eta)$ to Eq.~\eqref{eq:approx_opt_3}, we have
\begin{equation}\label{eq:approx_opt_4}
\begin{aligned}
    -\langle b, y \rangle \leq &~ \mathrm{OPT} + (-\langle b, z(\eta) \rangle - \mathrm{OPT}) \cdot (1 + \|y - z(\eta))\|_{H(z(\eta))}) \\
    \leq &~ \mathrm{OPT} + \frac{\nu_F}{\eta} \cdot (1 + \|y - z(\eta))\|_{H(z(\eta))}),
\end{aligned}
\end{equation}
where the second step follows from Eq.~\eqref{eq:approx_opt_1}.

{\bf Step 5.}
Finally, for simplicity denote the Newton step of $F_{\eta}$ at $y$ as $n(y) := H(y)^{-1} g_{\eta}(y) \in \R^m$. We have $\|n(y)\|_{H(y)} \leq \epsilon_N$ where $\epsilon_N \leq 0.05$, and using Theorem~2.2.5 of \cite{renegarMathematicalViewInteriorPoint2001}, we have
\begin{equation*}
    \|y - z(\eta)\|_{H(y)} \leq \|n(y)\|_{H(y)} + \frac{3 \|n(y)\|_{H(y)}^2}{(1 - \|n(y)\|_{H(y)})^3} \leq \epsilon_N \cdot (1 + \frac{3 \epsilon_N}{(1 - \epsilon_N)^3}) \leq 1.5 \epsilon_N.
\end{equation*}
Then by the definition of self-concordance (Section~2.2.1 of \cite{renegarMathematicalViewInteriorPoint2001}), we have
\begin{equation*}
    \|y - z(\eta)\|_{H(z(\eta))} \leq \frac{\|y - z(\eta)\|_{H(y)}}{1 - \|y - z(\eta)\|_{H(y)}} \leq \frac{1.5 \epsilon_N}{1 - 1.5 \epsilon_N} \leq 2 \epsilon_N.
\end{equation*}

Thus, from Eq.~\eqref{eq:approx_opt_4} we have
\begin{equation*}
    -\langle b, y \rangle \leq \mathrm{OPT} + \frac{\nu_F}{\eta} \cdot (1 + 2 \epsilon_N). \qedhere
\end{equation*}
\end{proof}

\subsection{Low rank update}
We use the following low rank update procedure of \cite{jiang2020faster} and \cite{hjst21}, which we modify by using a cutoff when $r\geq U/L$.
The proof of the following lemma can be found in \cite[Theorem~10.8]{hjst21}.
\begin{algorithm}[htb!]
    \caption{\textsc{LowRankUpdate} of \cite{jiang2020faster}}
    \label{alg:low_rank_update}
    \SetKwInOut{Input}{Input}
    \SetKwInOut{Parameters}{Parameters}
    \SetKwInOut{Output}{Output}
    \SetKw{And}{\textbf{and}}
    \Parameters{A real number $\epsilon_S < 0.01$.}
    \Input{New exact matrix $S^{\new} \in \R^{L \times L}$, old approximate matrix $\wt{S} \in \R^{L \times L}$, }
    \Output{New approximate matrix $\wt{S}^{\new} \in \R^{L \times L}$, update matrices $V_1, V_2 \in \R^{L \times r}$.} 

    $Z_{\midd} \gets (S^{\new})^{-1/2} \wt{S} (S^{\new})^{-1/2}  - I$
    
    Compute spectral decomposition $Z_{\midd} = X \diag(\lambda) X^\top$
    
    Let $\pi: [L] \to [L]$ be a sorting permutation such that $|\lambda_{\pi(i)}| \ge |\lambda_{\pi(i+1)}|$.
    
    \If{$|\lambda_{\pi(1)}| \le \epsilon_S$}{
        $\wt{S}_{\new} \gets \wt{S}$\;
        
        \Return{$(\wt{S}^{\new}, 0, 0)$}
    }
    \Else{
        $r \gets 1$\;
        
        \While{$2r \leq U/L$ \emph{and} $(|\lambda_{\pi(2r)}| > \epsilon_S$ \emph{or} $|\lambda_{\pi(2r)}| > (1 - 1 / \log L) |\lambda_{\pi(r)}|)$}{
            $r \gets r + 1$\;
            }
        $r \leftarrow 2 r$\;
        
        \If{$r \ge U/L$}{
            $\wt{S}^{\new} \gets S^{\new}$ \tcp*{Here we deviate from \cite{jiang2020faster}}
            
            \Return{$(\wt{S}^{\new}, \mathtt{null}, \mathtt{null})$}
        }
        \Else{
            $\lambda^{\new}_{\pi(i)} \gets 
            \begin{cases} 
            0 & \text{if } i = 1,2, \ldots, r \\ 
            \lambda_{\pi(i)} & \text{else }
            \end{cases}$
            
            $\Omega \leftarrow \supp(\lambda^{\new} - \lambda)$ \tcp*{$|\Omega| = r$}
            
            $V_1 \leftarrow ((S^{\new})^{1/2} \cdot X \cdot \diag (\lambda^{\new} - \lambda) )_{:,\Omega}$ \tcp*{$V_1 \in \R^{L \times r}$}
            
            $V_2 \leftarrow ((S^{\new})^{1/2} \cdot X )_{:,\Omega}$ \tcp*{$V_2 \in \R^{L \times r}$}
            
            $\wt{S}^{\new} \gets \wt{S} + (S^{\new})^{1/2} X \diag(\lambda^{\new} - \lambda) X^\top (S^{\new})^{1/2}$ \tcp*{$\wt{S}^{\new} = \wt{S} + V_1 V_2^{\top} \in \R^{L \times L}$}
            
            \Return{$(\wt{S}^{\new}, V_1, V_2)$}\;
        }
    }
  \end{algorithm}
\begin{lemma}[Low rank update]\label{lem:low_rank_update}
The algorithm $\textsc{LowRankUpdate}$ (Algorithm~\ref{alg:low_rank_update}) has the following properties:
\begin{enumerate}[(i)]
    \item The output matrix $\wt{S}^{\new} = \wt{S} + V_1 V_2^{\top}$ is a spectral approximation of the input matrix: \[
    \wt{S}^{\new} \approx_{\epsilon_S} S^{\new}.
    \]
    \item \label{it:general-non-increasing} Consider $t$ iterations of \textsc{LowRankUpdate}. Initially $\wt{S}^{(0)} = S^{(0)}$, and we use $(S^{(i)}, \wt{S}^{(i-1)})$ and $(\wt{S}^{(i)}, V_1^{(i)}, V_2^{(i)})$ to denote the input and the output of the $i$-th iteration. We define the rank $r_i$ to be the rank of $V_1^{(i)}$ if $V_1^{(i)} \neq \mathtt{null}$, and otherwise we define $r_i = U/L$.
    
    If the input exact matrices $S^{(0)}, S^{(1)}, \cdots, S^{(t)} \in \R^{L \times L}$ satisfy
    \begin{equation}\label{eq:slow_changing_requirement}
        \|(S^{(i-1)})^{-1/2} S^{(i)} (S^{(i-1)})^{-1/2} - I\|_F \leq 0.02, ~\forall i \in [t].
    \end{equation}
    Then for any non-increasing sequence $g \in \R^L_+$, the ranks $r_i$ satisfy
    \begin{equation*}
        \sum_{i=1}^t r_i \cdot g_{r_i} \leq O(t \cdot \|g\|_2 \cdot \log L).
    \end{equation*}
\end{enumerate}
Furthermore, the algorithm \textsc{LowRankUpdate} takes $O(L^{\omega})$ time.
\end{lemma}

\subsection{Slowly moving guarantee}
In SOS optimization, the matrix $S = P^{\top} \diag(s) P$ corresponds to the slack matrix of the SDP. The following lemma proves similar to SDP, in SOS the matrix $S$ is changing slowly. The proof is deferred to the appendix. Using this lemma we will prove that the requirement Eq.~\eqref{eq:slow_changing_requirement} of Lemma~\ref{lem:low_rank_update} is satisfied, which means we can approximate the change to the slack by a low-rank matrix.
\begin{lemma}[Slowly moving guarantee]\label{lem:slowly_moving}
Let $c \in \R^U$ and $A \in \R^{m \times U}$ be the input to the optimization problem. Let $P \in \R^{U \times L}$ be the matrix of the interpolant basis. 

For any $y \in \R^m$ and $y^{\new} = y + \delta_y \in \R^m$, let $s = c - A^{\top} y \in \R^U$ and $S = P^{\top} \diag(s) P \in \R^{L \times L}$. Similarly define $s^{\new}$ and $S^{\new}$ from $y^{\new}$. Let $H(y) = A \cdot \big( P \big( P^{\top} \diag(s) P \big)^{-1} P^{\top} \big)^{\circ 2} \cdot A^{\top} \in \R^{m \times m}$. If $s, s^{\new} \in \Sigma_{n,2d}^*$, then $S$ and $S^{\new}$ are both PSD, and we have
\begin{equation*}
    \|S^{-1/2} S^{\new} S^{-1/2} - I\|_F = \|\delta_y\|_{H(y)}.
\end{equation*}
\end{lemma}
\begin{proof}
Note that if $s, s^{\new} \in \Sigma_{n,2d}^*$, then by the dual cone characterization (Theorem~\ref{prop:nesterov_dual_SOS}) $S$ and $S^{\new}$ are both PSD.

For convenience we define $M = P S^{-1} P^{\top} \in \R^{U \times U}$. Note that $H(y) = A \cdot M^{\circ 2} \cdot A^{\top}$. We also define $\delta_s = s^{\new} - s = -A^{\top} \delta_y$. $\forall u \in [U]$, we use $p_u \in \R^L$ to denote the $u$-th row of $P$.
\begin{equation}\label{eq:slowly_moving_1}
\begin{aligned}
\|S^{-1/2} S^{\new} S^{-1/2}  - I\|_F^2 &= \|S^{-1/2} \big( S^{\new} - S \big) S^{-1/2} \|_F^2 \\
&= \|S^{-1/2} \big( P \diag(\delta_s) P^{\top} \big) S^{-1/2} \|_F^2 \\
&=\tr\big(S^{-1}(P^\top \diag(\delta_s) P)S^{-1}(P^\top \diag(\delta_s) P)\big) \\
&= \tr\big(S^{-1}(\sum_{u \in U}(\delta_s)_u \cdot p_u p_u^\top)S^{-1}
(\sum_{v \in U}(\delta_s)_v \cdot p_v p_v^\top)
\big) \\
&= \sum_{u,v \in U} (\delta_s)_u (\delta_s)_v \cdot \tr\big(S^{-1} p_u p_u^\top S^{-1}
p_v p_v^\top
\big) \\
&= \sum_{u,v \in U} (\delta_s)_u (\delta_s)_v \cdot (p_v^\top S^{-1} p_u)^2 \\
&= \sum_{u,v \in U} (\delta_s)_u (\delta_s)_v \cdot M_{uv}^2
~=~ \|\delta_s\|_{M^{\circ 2}}^2,
\end{aligned}
\end{equation}
where the third step follows from $\|A\|_F^2 = \tr(A^{\top} A)$ and the cyclic property of trace, and the sixth step again follows from the cyclic property of trace.

Since $\delta_s = - A^{\top} \delta_y$, we have
\begin{equation}\label{eq:slowly_moving_2}
    \|\delta_s\|_{M^{\circ 2}}^2 = \delta_s^{\top} M^{\circ 2} \delta_s 
    = \delta_y^{\top} A M^{\circ 2} A^{\top} \delta_y 
    = \delta_y^{\top} H(y) \delta_y 
    = \|\delta_y\|_{H(y)}^2.
\end{equation}
Combining Eq.~\eqref{eq:slowly_moving_1} and \eqref{eq:slowly_moving_2} we get the bound in the lemma statement. 
\end{proof}

\subsection{Proof of correctness}
Finally we are ready to prove the correctness of Algorithm~\ref{alg:barrier}.
\begin{theorem}[Correctness of Algorithm~\ref{alg:barrier}]\label{thm:correct}
Consider the following optimization problem with $A \in \R^{m \times U}$, $b \in \R^m$, and $c \in \R^U$: 
\begin{equation*}
  \begin{aligned}
  \text{Primal:~~~} \min \; &\pr{c}{x} \quad \\
  \mathrm{s.t.}~~ Ax& =b \\
  x &\in \Sigma_{n,2d}\, ,\\
  \end{aligned}
  \quad\quad\quad
  \begin{aligned} 
  \text{Dual:~~~}\max \; & \pr{y}{b} \\
  \mathrm{s.t.}~~ A^\top y + s &= c \\
  s &\in \Sigma_{n,2d}^*\, . \\
  \end{aligned}
\end{equation*}
Let $\mathrm{OPT}$ denote the optimal objective value of this optimization problem.
Assume Slater's condition and that any primal feasible $x \in \Sigma_{n,2d}$ satisfies $\|x\|_1 \leq R$.

Then for any error parameters $\delta \in (0,1)$, $\epsilon_S \leq 0.01$, and $\epsilon_N \leq 0.05$, Algorithm~\ref{alg:barrier} outputs $x \in \Sigma_{n,2d}$ that satisfies
\begin{equation*}
    \pr{c}{x} \le \mathrm{OPT} + \delta \cdot R \|c\|_\infty \quad \text{and} \quad
    \|A x - b\|_1 \le 8\delta L \cdot (LR \|A\|_\infty + \|b\|_1).
\end{equation*}
\end{theorem}
\begin{proof}
We consider the optimization problem $\min -b^{\top} y$ s.t. $y \in \ov{D}_F$, where $F: \R^m \to \R_+$ is the barrier function defined in Eq.~\eqref{eq:barrier}, and $\ov{D}_F$ is the closure of the domain of $F$. The barrier parameter of $F$ is $\nu_F = L$. This optimization problem is equivalent to the dual formulation and its optimal value is $-\mathrm{OPT}$. For any $\eta$, let $F_{\eta}(y) = -\eta b^{\top} y + F(y)$.

In the beginning Algorithm~\ref{alg:barrier} first uses Lemma~\ref{lem:init} to convert the optimization problem to another form which has an initial feasible solution $y$ that is close to the optimal solution of $F_{\eta}$ with $\eta=1$. The initial $y$ satisfies $\|g_{\eta}(y)\|_{H(y)^{-1}} \leq \epsilon_N$ by \Cref{lem:init}. Initially we also have $\wt{S} = S = P^{\top} \diag(s) P$ (Line~\ref{line:initial_S} in Algorithm~\ref{alg:barrier}).

Next we prove the correctness of Algorithm~\ref{alg:barrier} inductively. At each iteration, we assume the following induction hypothesis is satisfied: (1) $\|g_{\eta}(y)\|_{H(y)^{-1}} \leq \epsilon_N$, (2) $\wt{S} \approx_{\epsilon_S} S$.
We aim to prove that the updated $y^{\new}$, $\eta^{\new}$, $S^{\new}$, and $\wt{S}^{\new}$ still satisfy these two conditions.

In Lemma~\ref{lem:hessian_inverse_update} we have proved that in Algorithm~\ref{alg:barrier} we always maintain $N = \big( A \cdot (P \wt{S}^{-1} P^{\top})^{\circ 2} \cdot A^{\top} \big)^{-1}$. Let $\wt{H} = N^{-1}$, we have
\begin{equation*}
    \wt{H} = A \cdot \big( P \wt{S}^{-1} P^{\top} \big)^{\circ 2} \cdot A^{\top} \approx_{2 \epsilon_S} A \cdot \big( P S^{-1} P^{\top} \big)^{\circ 2} \cdot A^{\top} = H(y),
\end{equation*}
where in the second step we use the induction hypothesis that $\wt{S} \approx_{\epsilon_S} S$, and by Fact~\ref{fac:spectral_approx} we have $\wt{S}^{-1} \approx_{\epsilon_S} S^{-1}$, and hence $P \wt{S}^{-1} P^{\top} \approx_{\epsilon_S} P S^{-1} P^{\top}$, and hence $(P \wt{S}^{-1} P^{\top})^{\circ 2} \approx_{2 \epsilon_S} (P S^{-1} P^{\top})^{\circ 2}$.

The new vector $y^{\new}$ is computed as $y^{\new} = y + \delta_y$ where $\delta_y = - \wt{H}^{-1} g_{\eta}(y)$ (Line~\ref{line:delta_y} and \ref{line:y_new} of Algorithm~\ref{alg:barrier}). And $\eta$ is updated to $\eta^{\new} = \eta \cdot (1 + \frac{\epsilon_N}{20 \sqrt{L}})$ (Line~\ref{line:eta} of Algorithm~\ref{alg:barrier}). Since $\|g_{\eta}(y)\|_{H(y)^{-1}} \leq \epsilon_N$, and $\wt{H} \approx_{2 \epsilon_S} H(y)$ where $2 \epsilon_S \leq 0.02$ by its definition in Algorithm~\ref{alg:low_rank_update}, the requirements of Lemma~\ref{lem:invariant_newton} are satisfied, so we have
\begin{equation*}
    \|g_{\eta^{\new}}(y^{\new})\|_{H(y^{\new})^{-1}} \leq \epsilon_N, ~~~\text{and} ~~~~\|\delta_y\|_{H(y)} \leq 2 \epsilon_N.
\end{equation*}
This proves the first induction hypothesis.

Then using Lemma~\ref{lem:slowly_moving} and since $\epsilon_N \leq 0.01$ by its definition in Algorithm~\ref{alg:barrier}, we have
\begin{equation*}
    \|S^{-1/2} (S^{\new}) S^{-1/2} - I\|_F \leq \|\delta_y\|_{H(y)} \leq 2 \epsilon_N \leq 0.02.
\end{equation*}
Thus the input matrix $S^{\new}$ to \textsc{LowRankUpdate} satisfies the requirement of Eq.~\eqref{eq:slow_changing_requirement} of Lemma~\ref{lem:low_rank_update}, and we have that $\wt{S}^{\new} \approx_{\epsilon_S} S^{\new}$. This proves the second induction hypothesis.

Finally, we know that after $t = 40 \epsilon_N^{-1} \sqrt{L} \log(L/\delta)$ iterations, $\eta$ becomes $(1+\frac{\epsilon_N}{20 \sqrt{L}})^{t} \geq 2L / \delta^2$, so using Lemma~\ref{lem:approximate_optimality}, we have
\begin{equation*}
    b^{\top} y \geq \mathrm{OPT} - \frac{\nu_F}{\eta} \cdot (1 + 2 \epsilon_N) \geq \mathrm{OPT} - \delta^2.
\end{equation*}
Thus the initialization lemma (Lemma~\ref{lem:init}) ensures that we have a solution $x \in \Sigma_{n,2d}$ to the original primal optimization problem which satisfies
\begin{equation*}
    \pr{c}{x} \le \mathrm{OPT} + \delta \cdot R \|c\|_\infty, ~~~~
    \|A x - b\|_1 \le 8\delta L \cdot (LR \|A\|_\infty + \|b\|_1). \qedhere
\end{equation*}
\end{proof} \section{Time complexity}

\subsection{Worst case time}
We first bound the worst case running time of Algorithm~\ref{alg:barrier}. The running time of the $i$-th iteration depends on the updated rank $r_i$ of \textsc{LowRankUpdate}, which is defined to be the size of $V_1^{(i)}$ if $V_1^{(i)} \neq \mathtt{null}$, and $U/L$ otherwise (see Lemma~\ref{lem:low_rank_update}).
\begin{lemma}[Worst case time of Algorithm~\ref{alg:barrier}]\label{lem:worst_case_time}
In Algorithm~\ref{alg:barrier}, the initialization time is $O(U^{\omega})$, and the running time in the $i$-th iteration is $O(\Tmat(U,U,\min\{Lr_i,U\}))$.
\end{lemma}
\begin{proof}
{\bf Initialization time.} The most time-consuming step of initialization is Line~\ref{line:N_init}, where computing $N = \big( A (P T P^{\top})^{\circ 2} A^{\top} \big)^{-1}$ takes $O(\Tmat(U,U,L) + \Tmat(U,U,m))$ time. This is bounded by $O(U^{\omega})$ since $L, m \leq U$.

{\bf Time per iteration.} In each iteration the most time-consuming steps are (1) computing $S^{\new}$ and calling \textsc{LowRankUpdate} on Line~\ref{line:compute_S_new}-\ref{line:low_rank_update}, (2) executing the if-clause on  Line~\ref{line:if_start}-\ref{line:if_end}, and (3) computing $g^{\new}$ on Line~\ref{line:g}.

\begin{enumerate}
    \item Computing $S^{\new}$ on Line~\ref{line:compute_S_new} takes $\Tmat(U,L,L)$ time. Calling $\textsc{LowRankUpdate}$ on Line~\ref{line:low_rank_update} takes $O(L^{\omega})$ time by Lemma~\ref{lem:low_rank_update}.
    \item In the if-clause on  Line~\ref{line:if_start}-\ref{line:if_end}, if $V_1 = V_2 = \mathtt{null}$, then $L r_i \geq U$, and we compute $N^{\new} = \big( A \cdot (P T^{\new} P^{\top})^{\circ 2} \cdot A^{\top} \big)^{-1}$, which takes $O(\Tmat(U,U,U))$ time. Otherwise we call \textsc{UpdateHessianInv} on Line~\ref{line:hessian_inverse_update}, which takes $O(\Tmat(U,U,L r_i))$ time by Lemma~\ref{lem:hessian_inverse_update}. In total the if-clause has running time $O(\Tmat(U,U, \min\{L r_i, U\}))$.
    \item Computing the gradient $g = -\eta^{\new} \cdot b + A \cdot \diag\Big( P \big( P^{\top} \diag(s^{\new}) P \big)^{-1} P^{\top} \Big)$ on Line~\ref{line:g} takes $O(\Tmat(U,U,L))$ time since $m \leq U$.
\end{enumerate}

Thus the total time per iteration is $O(\Tmat(U,U, \min\{L r_i, U\}))$.
\end{proof}

\subsection{Amortized time}\label{sec:amortize}
In this section we bound the amortized running time of Algorithm~\ref{alg:barrier}. 

Let $\omega$ be the matrix multiplication exponent, let $\alpha$ be the dual matrix multiplication exponent. The current best values are $\omega \approx 2.373$ and $\alpha \approx 0.314$ \cite{l14, gu18, aw21}. Note that the current best values of $\omega$ and $\alpha$ satisfies that $\alpha \geq 5 - 2 \omega$.
We use the following modified lemma from \cite{hjst21}:
\begin{lemma}[Helpful lemma for amortization, modified version of Lemma~10.13 of \cite{hjst21}]\label{lem:general_amortize_tool}
Let $t$ denote the total number of iterations. Let $r_i \in [L]$ be the rank for the $i$-th iteration for $i \in [t]$. Assume $r_i$ satisfies the following condition:
for any vector $g \in \R_+^L$ which is non-increasing, we have $\sum_{i=1}^t r_i \cdot g_{r_i} \leq O(t \cdot \|g\|_2)$.

If the cost in the $i$-th iteration is $O(\Tmat(U, U, \min\{L r_i, U\}))$, when $\alpha \geq 5 - 2 \omega$, the amortized cost per iteration is $U^{2 + o(1)} + U^{\omega - 1/2 + o(1)} \cdot L^{1/2}$.
\end{lemma}
For completeness we provide a proof of this lemma in Section~\ref{sec:proof_amortize}.

\begin{theorem}[Time of Algorithm~\ref{alg:barrier}]
When $\alpha \geq 5 - 2 \omega$, the running time of Algorithm~\ref{alg:barrier} is
\[
(U^{2} \cdot L^{1/2} + U^{\omega - 1/2} \cdot L) \cdot (\log(1/\delta) + U^{o(1)}).
\]
\end{theorem}
\begin{proof}
Using Lemma~\ref{lem:worst_case_time} the initialization time is $O(U^{\omega}) \leq O(U^{\omega - 1/2} \cdot L)$ since $U \leq L^2$.

Using Lemma~\ref{lem:low_rank_update} we know that the ranks $r_i$ indeed satisfy the requirement of Lemma~\ref{lem:general_amortize_tool}, and since the worst case time per iteration is $O(\Tmat(U,U,\min\{Lr_i, U\}))$ (Lemma~\ref{lem:worst_case_time}), using Lemma~\ref{lem:general_amortize_tool} the time per iteration is $U^{2 + o(1)} + U^{\omega - 1/2 + o(1)} \cdot L^{1/2}$. Since there are in total $t = 40 \epsilon_N^{-1} \sqrt{L} \log(L/\delta)$ iterations, we get the total running time claimed in the lemma statement.
\end{proof}

\subsection{Comparison with previous results}
In this section we compare the running time of \cite{py19}, \cite{jiang2020faster,hjst21}, and our result. We assume that $m = \Theta(U)$ when making the comparisons.

Ignoring $\log(1/\delta)$ and $U^{o(1)}$ factors, and since $L \le U \le L^2$, the running times are
\begin{equation*}
\begin{aligned}
\text{\cite{py19} (SOS)}: &~ L^{0.5} \cdot U^{\omega}, \\
\text{\cite{jiang2020faster, hjst21} (SDP)\footnotemark}: &~ L^{0.5} \cdot \min\{UL^2 + U^{\omega},~ L^4 + L^{2 \omega - 0.5}\}, \\
\text{Ours (SOS)}: &~ L^{0.5} \cdot (U^2 + U^{\omega - 0.5} \cdot L^{0.5}).
\end{aligned}
\end{equation*}
\footnotetext{When solving SOS, \cite{jiang2020faster} has running time $O(L^{0.5} \cdot (U L^2 + U^{\omega} + L^{\omega})) \leq O(L^{0.5} \cdot (U L^2 + U^{\omega}))$, and \cite{hjst21} has running time $O(L^{0.5} \cdot (U^2 + L^4) + U^{\omega} + L^{2 \omega}) \leq O(L^{4.5} + L^{2 \omega})$ since $L \leq U \leq L^2$.}
\vspace{-5mm}

\paragraph*{Current $\omega$ and $\alpha$.} Plugging in the current best values $\omega \approx 2.373$ and $\alpha \approx 0.314$, we have
\begin{equation*}
\begin{aligned}
\text{\cite{py19} (SOS)}: &~ L^{0.5} \cdot U^{2.373}, \\
\text{\cite{jiang2020faster, hjst21} (SDP)}: &~ L^{0.5} \cdot \min\{ UL^2 + U^{2.373}, ~ L^{4.246} \} \\
= &~ L^{0.5} \cdot \begin{cases}
UL^2 & \text{when~} U\in (L, L^{1.457}], \\
U^{2.373} & \text{when~} U\in (L^{1.457}, L^{1.789}],\\
L^{4.246} & \text{when~} U\in (L^{1.789}, L^2),
\end{cases}
\\
\text{Ours (SOS)}: &~ L^{0.5} \cdot (U^{2} + U^{1.873} L^{0.5}).
\end{aligned}
\end{equation*}
Note that our running time is always better than the previous results, and for several values of $L$ and $U$ we improve by a polynomial factor. See Figure~\ref{fig:runtimes} for an illustration. \section{Weighted SOS}\label{sec:wsos}

In this section we provide the background of weighted sum-of-squares (WSOS) optimization, following the notation of \cite{py19}. 

Recall the motivation for sum-of-squares to solve polynomial optimization in \eqref{eq_poly_opt}. Given $k$ non-zero polynomials $f_1, \ldots, f_k$, we aim to optimize over the polynomials that are non-negative on $\mathcal{S} := \{x \in \R^n \mid f_i(x) \ge 0, \forall i \in [k]\}$. While polynomial optimization is hard in general, we can relax the problem to optimize over the following set of non-negative polynomials over $\mathcal{S}$. 
\begin{definition}[WSOS polynomials]\label{def:wsos}
Let $\boldf = (f_1, \ldots, f_k)$ where $f_i$ are non-zero polynomials, and let $\bd = (d_1,\ldots, d_k)$ where $d_i > 0$. Define $\V_{n, 2 \bd}^{\boldf} := \{\sum_{i=1}^k f_i s_i \mid s_i \in \V_{n, 2d_i}, \forall i \in [k]\}$. The WSOS polynomials are defined as
\begin{equation*}
\Sigma_{n,2\bd}^{\boldf} := \Big\{\sum_{i = 1}^k f_i s_i \mid s_i \in \Sigma_{n, 2 d_i}, \forall i \in [k]\Big\} \subset \V_{n,2\bd}^{\boldf}.
\end{equation*}
\end{definition}
We assume that the set $\mathcal{S}$ is unisolvent for $\V_{n, 2\bd}^{\boldf}$, so that $\Sigma_{n,2\bd}^{\boldf}$ is a pointed and closed cone in $\V_{n,2\bd}^{\boldf}$. (See Proposition 6.1 of \cite{py19}.) Note that w.l.o.g.\ one can add $f_{k+1} = 1$ to $\boldf$ to capture the SOS polynomials. Define 
\[
U := \dim(\V_{n, 2\bd}^{\boldf}), ~~\text{and}~~ L_i := \dim(\mathcal V_{n,d_i}) = {n + d_i \choose n}, \forall i \in [k].
\]
Note that $U \geq L_i$ for all $i \in [k]$. We also define $L = \max_i L_i$. Similar to SOS optimization, after fixing a basis $(q_1, q_2, \cdots, q_U)$ of $\V_{n, 2\bd}^{\boldf}$, a polynomial in $\V_{n, 2\bd}^{\boldf}$ corresponds to the vector of coefficients in $\R^U$.

We consider the following optimization problem over the WSOS cone ($A \in \R^{m \times U}$, $c \in \R^U$, $b \in \R^m$):
\begin{equation}
  \label{WSOS_program_primal_dual}
  \tag{WSOS}
  \begin{aligned}
  \text{Primal:~~~} \min \; &\pr{c}{x} \quad \\
  \mathrm{s.t.}~~ Ax& =b \\
  x &\in \Sigma_{n,2\bd}^{\boldf}\, ,\\
  \end{aligned}
  \quad\quad\quad
  \begin{aligned} 
  \text{Dual:~~~}\max \; & \pr{y}{b} \\
  \mathrm{s.t.}~~ A^\top y + s &= c \\
  s &\in \Sigma_{n,2\bd}^{\boldf \; *}\, . \\
  \end{aligned}
\end{equation}
Here $\Sigma_{n,2\bd}^{\boldf \; *}$ denotes the dual cone of $\Sigma_{n,2\bd}^{\boldf}$. 

Similar to \Cref{prop:nesterov_dual_SOS}, the dual cone $\Sigma_{n,2\bd}^{\boldf \; *}$ again admits a semidefinite characterisation. For any basis $\mathbf{q} = (q_1, q_2, \cdots, q_U)$ of $\V_{n, 2\bd}^{\boldf}$, and any basis $\mathbf{p}_i = (p_{i, 1}, \ldots, p_{i, L_i})$ of $\mathcal{V}_{n,d_i}$ for $i \in [k]$, let $\Lambda_i: \R^U \to \R^{L_i \times L_i}$ be the unique linear mapping satisfying $\Lambda_i(\mathbf{q}) = \mathbf{p}_i \mathbf{p}_i^\top$. Then the dual cone admits the characterization
\begin{equation}\label{eq:wsos_dual_characterization}
    \Sigma_{n,2\bd}^{\boldf \; *} = \{s \in \R^U \mid \Lambda_i(s) \succeq 0, \forall i \in [k]\}.
\end{equation}
For more details see Theorem 17.6 of \cite{nesterov2000squared}.

\paragraph*{Barrier function under interpolant basis}
\cite{py19} extends the interpolant basis to WSOS. Let $\mathcal{T} = \{t_1, \cdots, t_U\} \subset \R^n$ be a unisolvent set for $\V_{n, 2\bd}^{\boldf}$. Let $\mathbf{q} = (q_1, q_2, \cdots, q_U)$ be the Lagrange polynomials corresponding to $\mathcal{T}$, and $\mathbf{q}$ forms a basis of $\V_{n, 2\bd}^{\boldf}$. For all $i \in [k]$, choose any basis $\mathbf{p}_i = (p_{i, 1}, \ldots, p_{i, L_i})$ of $\mathcal{V}_{n,d_i}$. It was shown in \cite{py19} that the unique linear mapping $\Lambda_i: \R^U \to \R^{L_i \times L_i}$ satisfying $\Lambda_i(\mathbf{q}) = \mathbf{p}_i \mathbf{p}_i^\top$ is
\begin{equation}
\Lambda_i(s): = P_i^\top \diag(\boldf_i \circ s)P_i.
\end{equation}
where $P_i \in \R^{U \times L_i}$ is defined as $(P_i)_{u,\ell} := p_{i,\ell}(t_u)$, and $\boldf_i := (f_i(t_1), \ldots, f_i(t_U)) \in \R^U$. 

Define $F_i: \R^m \to \R$ as $F_{i}(y) = - \log \det(\Lambda_i(c - A^{\top} y))$. We get the following barrier with barrier parameter at most $\sum_{i = 1}^k L_i$:
\begin{equation*}
F^\boldf(y) = \sum_{i=1}^k F_i(y) = - \sum_{i = 1}^k \log \det(\Lambda_i(c - A^{\top} y)).
\end{equation*}
For any $\eta >0$ we define a function $F_{\eta}^\boldf: \R^m \to \R$ as $F_{\eta}(y) = -\eta \cdot b^{\top} y + F^\boldf(y)$.

The corresponding gradients and Hessians of the $F_i$'s are:
\begin{equation*}
\begin{aligned}
    g_{i}(y) &= A\cdot \Big(\boldf_i \circ \diag(P_i(P_i^\top \diag(\boldf_i \circ s)P_i)^{-1}P_i^\top)\Big), \\
    H_{i}(y) &= A \cdot \Big((\boldf_i\boldf_i^\top) \circ (P_i(P_i^\top \diag(\boldf_i \circ s) P_i)^{-1}P_i^\top)^{\circ 2}\Big) \cdot A^\top,
\end{aligned}
\end{equation*}
and so gradient and Hessian of $F^\boldf_\eta(y)$ are
\begin{equation}\label{eq:wsos_barrier_gradient_hessian}
g^\boldf_{\eta}(y) = -\eta b + \sum_{i=1}^k g_{i}(y), ~~\text{and}~~H^\boldf(y) = \sum_{i=1}^k H_{i}(y).
\end{equation}
Note that we omit the subscript $\eta$ for the Hessian since it is independent of $\eta$.

\subsection{Extension of our SOS algorithm to WSOS}\label{sec:wsos_algorithm}
We next show how to extend our SOS algorithm (Algorithm~\ref{alg:barrier}) to WSOS. The main adjustments are the following.
\begin{enumerate}
    \item {\bf Maintain $k$ approximate slack.} For each $i \in [k]$, we maintain a matrix $\wt{S}_i \in \R^{L_i \times L_i}$ to approximate the exact matrix $S_i := \Lambda_i(s) = P_i^\top \diag(\boldf_i \circ s) P_i \in \R^{L_i \times L_i}$. We also maintain $T_i = \wt{S}_i^{-1}$. In each iteration we compute $S_i^{\new} = P^{\top} \diag(\boldf_i \circ s^{\new}) P_i \in \R^{L_i \times L_i}$ 
    \item {\bf \textsc{LowRankUpdate} on diagonal block matrices.} In the $j$-th iteration, to compute the low-rank update to the $\wt{S}_i$'s, we call \textsc{LowRankUpdate} (Algorithm~\ref{alg:low_rank_update}) with the two block diagonal matrices $\diag(S_1^{\new}, \cdots, S_k^{\new})$ and $\diag(\wt{S}_1, \cdots, \wt{S}_k)$ as inputs. We do not need to explicitly construct these two block diagonal matrices. It's easy to check that in \textsc{LowRankUpdate} all the computations maintain the block diagonal structure. Thus the outputs of \textsc{LowRankUpdate} are also block diagonal matrices, and for all $i \in [k]$ we read off $\wt{S}_i^{\new} \in \R^{L_i \times L_i}$, and the low-rank updates $V_{i,1}, V_{i,2} \in \R^{L_i \times r_{i,j}}$.

    \item {\bf Extend \textsc{UpdateHessianInv} to WSOS.} We extend $\textsc{UpdateHessianInv}$ (Algorithm~\ref{alg:hessian_inverse_update}), as shown in $\textsc{UpdateHessianInvWSOS}$ (Algorithm~\ref{alg:hessian_inverse_update_wsos}). Now we need to stack the $k$ updates of the $H_i$'s together, and perform a rank-$\sum_{i=1}^k r_{i,j}$ update. Using a similar proof as that of Lemma~\ref{lem:hessian_inverse_update}, we can show that 
    \begin{equation*}
    \begin{aligned}
    & T_i^{\new} = (\wt{S}_i^{\new})^{-1} \in \R^{L_i \times L_i},~~ \forall i \in [k], \\
    & N^{\new} = \Big( A \cdot \big(\sum_{i=1}^k (\boldf_i \boldf_i^{\top}) \circ (P_i (\wt{S}_i^{\new})^{-1} P_i^{\top})^{\circ 2} \big) \cdot A^{\top} \Big)^{-1} \in \R^{m \times m}.
    \end{aligned}
    \end{equation*}
\end{enumerate}

\begin{algorithm}[!ht]
    \caption{\textsc{UpdateHessianInvWSOS}}
    \label{alg:hessian_inverse_update_wsos}
    \SetKwInOut{Input}{Input}
    \SetKwInOut{Output}{Output}
    \SetKw{And}{\textbf{and}}
    \Input{$T_i \in \R^{L_i \times L_i}$ for all $i \in [k]$, $N \in \R^{m \times m}$, $V_{i,1},V_{i,2} \in \R^{L_i \times r_i}$ for all $i \in [k]$} 
    \Output{$T_i^{\new} \in \R^{L_i \times L_i}$ for all $i \in [k]$, $N^{\new} \in \R^{m \times m}$}
    
    \tcp{Step 1}
    
    \For{$i = 1, 2, \cdots, k$}{
    $\ov{V_{i,1}} \leftarrow - T V_{i,1} \cdot (I + V_{i,2}^{\top} T V_{i,1} )^{-1}$ \tcp*{$\ov{V_{i,1}} \in \R^{L_i \times r_i}$}
    
    $\ov{V_{i,2}} \leftarrow T V_{i,2}$ \tcp*{$\ov{V_{i,2}} \in \R^{L_i \times r_i}$}

    $T_i^{\new} \leftarrow T_i + \ov{V_{i,1}} \cdot \ov{V_{i,2}}^{\top}$ \tcp*{$T_i^{\new} \in \R^{L_i \times L_i}$}
    }
    
    \tcp{Step 2}
    
    \For{$i = 1, 2, \cdots, k$}{
    $Y'_i \leftarrow \diag(\boldf_i) \cdot [2 P_i T_i, P_i \ov{V_{i,1}}]$ \tcp*{$Y'_i \in \R^{U \times (L_i+r_i)}$}
    
    $Z'_i \leftarrow \diag(\boldf_i) \cdot [P_i, P_i \ov{V_{i,2}}]$ \tcp*{$Z'_i \in \R^{U \times (L_i+r_i)}$}
    
    $Y_i \leftarrow [\diag(u_{i,1})Y'_i, \cdots, \diag(u_{i,r})Y'_i]$, $u_{i,j}$ is $j$-th col. of $P_i \ov{V_{i,1}}$ \tcp*{$Y_i \in \R^{U \times (L_i+r_i) r_i}$}
    
    $Z_i \leftarrow [\diag(v_{i,1})Z'_i, \cdots, \diag(v_{i,r})Z'_i]$, $v_{i,j}$ is $j$-th col. of $P_i \ov{V_{i,2}}$ \tcp*{$Z_i \in \R^{U \times (L_i+r_i) r_i}$}
    }
    
    $Y \leftarrow [Y_1, \cdots, Y_k]$,~~ $Z \leftarrow [Z_1, \cdots, Z_k]$ \tcp*{$Y,Z \in \R^{U \times \sum_{i=1}^k(L_i + r_i) r_i}$}
    
    \tcp{Step 3}
    
    $N^{\new} \leftarrow N - N \cdot (AY) \cdot \big(I + (AZ)^{\top} N (AY) \big)^{-1} \cdot (AZ)^{\top} \cdot N$ \tcp*{$N^{\new} \in \R^{m \times m}$} \label{line:N_wsos}
    
    \Return{$T^{\new}, N^{\new}$}
\end{algorithm}

\subsection{Correctness}
In this section we prove the correctness of the WSOS algorithm described in the previous section. 
In the previous section we have shown that an analogue of Lemma~\ref{lem:hessian_inverse_update} holds for WSOS. Lemma~\ref{lem:invariant_newton} and Lemma~\ref{lem:approximate_optimality} directly hold for our WSOS algorithm since they can be applied to any barrier function. It remains to prove the following analogue of \Cref{lem:slowly_moving}.
\begin{lemma}[Slowly moving guarantee for WSOS]\label{lem:slow_move_wsos}
Let $s, s^{\new} \in \R^U$ be the current and the updated slack variables of the WSOS algorithm. For each $i \in [k]$, define $S_i = \Lambda_i(s) = P_i^{\top} \diag(\boldf_i \circ s) P_i$, and $S_i^{\new} = \Lambda_i(s^{\new}) = P_i^{\top} \diag(\boldf_i \circ s^{\new}) P_i$. Let $\ov{S} := \diag(S_1, \ldots, S_k)$, and $\ov{S}^{\new} := \diag(S_1^{\new}, \ldots, S_k^{\new})$. Then we have
\[
\|\ov{S}^{-1/2} \ov{S}^{\new} \ov{S}^{-1/2}  - I\|_F^2 = \|\delta_y\|_{H^\boldf(y)}^2,
\]
where $H^\boldf(y)$ is the Hessian of the barrier function $F^\boldf_\eta(y)$, as defined in Eq.~\eqref{eq:wsos_barrier_gradient_hessian}.
\end{lemma}
\begin{proof}
For each $i \in [k]$, we define $M_i := P_i S^{-1} P_i^\top$, and let $p_{i,u}$ denote the $u$-th row of $P_i$ for each $u \in [U]$. Let $\delta_s = -A^\top \delta y$, and note that $s^{\new} - s = \delta_s$. We have that
\begin{equation}
\begin{aligned}
\|S_i^{-1/2} S_i^{\new} S_i^{-1/2}  - I\|_F^2 &= \|S_i^{-1/2} \big( S_i^{\new} - S_i \big) S_i^{-1/2} \|_F^2 \notag \\
&= \|S_i^{-1/2} \big( P_i^{\top} \diag(\boldf_{i} \circ \delta_s) P_i \big) S_i^{-1/2} \|_F^2 \notag  \\
&=\tr\Big(S_i^{-1} \big(P_i^\top \diag(\boldf_i \circ \delta_s) P \big) S_i^{-1} \big(P_i^\top \diag(\boldf_i \circ \delta_s) P_i \big)\Big) \notag \\
&= \tr\Big(S_i^{-1} \big(\sum_{u \in U}(\boldf_{i})_u(\delta_s)_u \cdot p_{i,u} p_{i,u}^\top \big) S_i^{-1}
\big( \sum_{v \in U} (\boldf_{i})_v(\delta_s)_v \cdot p_{i,v} p_{i,v}^\top \big)
\Big) \notag \\
&= \sum_{u,v \in U} (\boldf_{i})_u (\boldf_{i})_v (\delta_s)_u (\delta_s)_v \cdot \tr\big(S_i^{-1} p_{i,u} p_{i,u}^\top S_i^{-1}
p_{i,v} p_{i,v}^\top
\big) \notag \\
&= \sum_{u,v \in U} (\boldf_{i})_u (\boldf_{i})_v  (\delta_s)_u (\delta_s)_v \cdot (p_{i,v}^\top S_i^{-1} p_{i,u})^2 \notag \\
&= \sum_{u,v \in U}(\boldf_{i})_u (\boldf_{i})_v  (\delta_s)_u (\delta_s)_v \cdot (M_i)_{uv}^2
~=~ \|\delta_s\|_{(\boldf_i \boldf_i^\top)\circ{M_i^{\circ 2}}}^2,
\end{aligned}
\end{equation}
where the third step follows from $\|A\|_F^2 = \tr(A^{\top} A)$ and the cyclic property of trace, and the sixth step again follows from the cyclic property of trace.

Since $\delta_s = -A^\top \delta y$ and $H_i(y) = A \big((\boldf_i \boldf_i^\top)\circ{M_i^{\circ 2}} \big) A^\top$, we have 
\begin{equation*}
\|\delta_s\|_{(\boldf_i \boldf_i^\top)\circ{M_i^{\circ 2}}} = \delta_s^\top \big((\boldf_i \boldf_i^\top)\circ{M_i^{\circ 2}}\big) \delta_s = \delta_y^\top A \big((\boldf_i \boldf_i^\top)\circ{M_i^{\circ 2}} \big) A^\top \delta_y = \|\delta_y\|^2_{H_i(y)},
\end{equation*}
and so  
\begin{equation*}
\|S_i^{-1/2} S_i^{\new} S_i^{-1/2}  - I\|_F^2 = \|\delta_y\|_{H_i(y)}^2.
\end{equation*}

Since $\ov{S} := \diag(S_1, \ldots, S_k)$, and $\ov{S}^{\new} := \diag(S_1^{\new}, \ldots, S_k^{\new})$, we have 
\begin{equation*}
\begin{aligned}
\|\ov S^{-1/2}\ov{S}^{\new}\ov S^{-1/2} - I\|_F = &~ \Big( \sum_{i=1}^k \|S_i^{-1/2}S_i^{\new} S_i^{-1/2} - I\|_F^2\Big)^{1/2} \\
= &~ \sum_{i=1}^k \|\delta_y\|_{H_i(y)}^2
= \|\delta_y\|_{H^{\boldf}(y)}^2,
\end{aligned}
\end{equation*}
where the last step follows from $H^{\boldf}(y) = \sum_{i=1}^k H_i(y)$.
\end{proof}

Using \Cref{lem:invariant_newton} with the WSOS barrier $F^{\boldf}$, we can bound $\|\delta_y\|_{H^{\boldf}(y)} \le 2\epsilon_N$ throughout the algorithm. Combining with \Cref{lem:slow_move_wsos}, the requirement of \Cref{lem:low_rank_update} is satisfied for the WSOS algorithm where the inputs to \textsc{LowRankUpdate} are block diagonal matrices $\diag(S_1^{\new}, \cdots, S_k^{\new})$ and $\diag(\wt{S}_1, \cdots, \wt{S}_k)$. In the $j$-th iteration let $r_{i,j}$ denote the rank of the update to $\wt{S}_i$, and let $r_j = \sum_{i=1}^k r_{i,j}$. Then for any vector $g \in \R_+^n$ which is non-increasing, we have 
\begin{equation}\label{eq:rank_wsos}
    \sum_{j=1}^t r_j g_{r_j} \le O(T \cdot \|g\|_2 \cdot \log L).
\end{equation}

Now we have proved all the analogous lemmas for the WSOS algorithm. The correctness of the WSOS algorithm follows from a similar argument at that of Theorem~\ref{thm:correct}. We summarize this in the following theorem.
\begin{theorem}[Correctness of WSOS Algorithm]\label{thm:correct_wsos}
Consider the following optimization problem with $A \in \R^{m \times U}$, $b \in \R^m$, and $c \in \R^U$: 
\begin{equation*}
  \begin{aligned}
  \text{Primal:~~~} \min \; &\pr{c}{x} \quad \\
  \mathrm{s.t.}~~ Ax& =b \\
  x &\in \Sigma_{n,2\bd}^{\boldf}\, ,\\
  \end{aligned}
  \quad\quad\quad
  \begin{aligned} 
  \text{Dual:~~~}\max \; & \pr{y}{b} \\
  \mathrm{s.t.}~~ A^\top y + s &= c \\
  s &\in \Sigma_{n,2\bd}^{\boldf \; *}\, . \\
  \end{aligned}
\end{equation*}
Let $\mathrm{OPT}$ denote the optimal objective value of this optimization problem.
Assume that any primal feasible $x \in \Sigma_{n,2\bd}^{\boldf}$ satisfies $\|x\|_1 \leq R$.

Then for any error parameters $\delta \in (0,1)$ and $\epsilon_S, \epsilon_N \leq 0.01$, the variant of Algorithm~\ref{alg:barrier} described in Section~\ref{sec:wsos_algorithm} outputs $x \in \Sigma_{n,2\bd}^{\boldf}$ that satisfies
\begin{equation*}
    \pr{c}{x} \le \mathrm{OPT} + \delta \cdot R \|c\|_\infty \quad \text{and} \quad
    \|A x - b\|_1 \le 8\delta L \cdot (LR \|A\|_\infty + \|b\|_1).
\end{equation*}
\end{theorem}

\subsection{Time analysis}
In this section we analyse the running time of the WSOS algorithm described in \Cref{sec:wsos_algorithm}. Recall that we define $L = \max_i L_i$.

\paragraph*{Worst-case time}
The worst-case time complexity of the WSOS algorithm can be analyzed in a similar way as that of \Cref{lem:worst_case_time}. We only mention the bottlenecks.
\begin{itemize}
    \item In initialization computing the Hessian inverse takes $O(k \Tmat(U,U,L) + U^{\omega})$ time.
    \item Similar to the SOS algorithm, in each iteration, there are two main bottlenecks:
    \begin{enumerate}
    \item The call to \textsc{UpdateHessianInvWSOS}, which takes $\Tmat(U,U,\min(U,Lr_j))$ time in iteration $j \in [t]$ 
    \item Computing all $P_i(P_i^\top \diag(\boldf_i \circ s^{\new})P_i)^{-1}P_i^\top$ when computing the gradient, which takes $k \Tmat(U,U,L)$ time.
\end{enumerate}
Thus the total time per iteration is $O(\Tmat(U,U,\min(Lr_j,U)) + k\Tmat(U,U,L))$. 
\end{itemize}

\paragraph*{Amortized time}
We can bound the amortized time of our WSOS algorithm using a similar argument as that of the SOS algorithm in \Cref{sec:amortize}. Again assume that $\alpha \geq 5 - 2 \omega$. Using \Cref{lem:general_amortize_tool} and \Cref{eq:rank_wsos} that we proved in the previous section, the amortized iteration complexity is
\begin{equation*}
    O\big(U^2 + U^{\omega - 1/2}L^{1/2} + k\Tmat(U,U,L)\big) \cdot (\log(1/\delta) + U^{o(1)}).
\end{equation*}
Overall we have $\sqrt{kL}$ iterations (since the barrier parameter is at most $\sum_{i=1}^k L_i \leq k L$), resulting in total complexity of 
\begin{equation*}
    O\big((kL)^{0.5}\cdot (U^2 + U^{\omega - 1/2}L^{1/2} + k\Tmat(U,U,L)) + U^{\omega} \big) \cdot (\log(1/\delta) + U^{o(1)}).
\end{equation*}

\paragraph*{Comparison with previous results}
Again assume that the number of constraints $m = \Theta(U)$.
The running times of \cite{py19}, \cite{jiang2020faster}, \cite{hjst21} and our algorithm are
\begin{equation*}
\begin{aligned}
\text{\cite{py19} (WSOS)}: &~ \tilde O\big((kL)^{0.5} \cdot (U^{\omega} + k \Tmat(U,U,L)) \big),\\
\text{\cite{jiang2020faster, hjst21} (SDP)} \footnotemark: &~
\tilde{O}\big( (kL)^{0.5} \cdot \min\{kU L^2 + U^{\omega} + (kL)^{\omega},~ U^2 + (kL)^4 + U^\omega/(kL)^{0.5} + (kL)^{2\omega - 0.5}\} \big), \\
\text{Ours (WSOS)}: &~ \tilde O\big((kL)^{0.5} \cdot (U^2 + U^{\omega-1/2} L^{1/2} + k \Tmat(U,U,L)) + U^{\omega}).
\end{aligned}
\end{equation*}
\footnotetext{When solving WSOS, \cite{jiang2020faster} has running time $O((kL)^{0.5} \cdot (kU L^2 + U^{\omega} + (kL)^{\omega}))$, and \cite{hjst21} has running time $O((kL)^{0.5} \cdot (U^2 + (kL)^4) + U^{\omega} + (kL)^{2 \omega})$. Here, we were able to drop a factor in $k$ due to the block-diagonal structure of the constraint matrix.}

\paragraph*{Reformulation as SDP}
We briefly explain the running time of WSOS when reformulated as SDP. Let $\bq = (q_1, q_2, \cdots, q_U)$ be a basis of $\mathcal V_{n,2\bd}^{\boldf}$. For each $i \in [k]$, let $\bp_{i} = (p_{i, 1}, \ldots, p_{i, L_i})$ be a basis of $\mathcal V_{n,d_i}$. Let $\Lambda_i: \R^U \to \R^{L_i \times L_i}$ be the unique linear mapping satisfying $\Lambda_i(\mathbf{q}) = \mathbf{p}_i \mathbf{p}_i^\top$. Let $\Lambda_i^*: \R^{L_i \times L_i} \to \R^U$ be the adjoint of $\Lambda$, i.e., $\Lambda_i^*$ is the unique linear operator that satisfies $\langle \Lambda_i(x), V \rangle = \langle x, \Lambda_i^*(V) \rangle$ for all $x \in \R^U$ and $V \in \R^{L_i \times L_i}$. Apart from the dual characterization described in Eq.~\eqref{eq:wsos_dual_characterization}, the WSOS cone also has a primal characterization (also proved in Theorem 17.6 of \cite{nesterov2000squared}):
\[
\Sigma_{n, 2\bd}^{\boldf} = \big\{x \in \R^U \mid x = \sum_{i=1}^k \Lambda_i^*(V_i),~~ V_i \succeq 0 \in \R^{L_i \times L_i}\text{ for } i \in [k]\big\}.
\]
We can therefore reformulate the primal of \eqref{WSOS_program_primal_dual} as the following SDP (we use $A_{j,*} \in \R^U$ to denote the $j$-th row of $A \in \R^{m \times U}$):

\begin{equation}
  \label{WSOS_SDP_reformulation}
  \tag{WSOS - SDP}
  \begin{aligned}
  \min \; &\sum_{i = 1}^k \pr{\Lambda_i(c)}{V_i} \quad \\
  \mathrm{s.t.}~~ &\sum_{i=1}^k \pr{\Lambda_i(A_{j,*})}{V_i} = b_j, \quad \forall j \in [m] \\
  &V_i \succeq 0,\quad \forall i \in [k].\\
  \end{aligned}
\end{equation}
Note that the number of constraints is still $m$, unchanged from $\eqref{WSOS_program_primal_dual}$, but the number of variables is now $\sum_{i=1}^k L_i^2 = O(kL^2)$, so just processing the constraint matrix in every iteration of an IPM algorithm costs up to $O(kUL^2)$. The SDP solvers \cite{jiang2020faster} and \cite{hjst21} can profit from the block-diagonal structure of the PSD matrices, which leads to the runtimes claimed in the previous paragraph. 

 \section{Discussion of Bottleneck}
\label{subsec:discussion}
The dominating term in the cost-per-iteration of our SOS algorithm is $\Tmat(U,U,Lr)$. A natural question is whether this can be improved to $\Tmat(U,U,L)$.

The bottleneck for this potential improvement is the following self-contained batch matrix-product problem, which is interesting in its own right:
\begin{itemize}
    \item \bf Input \rm: invertible matrix $H \in \R^{U \times U}$, and $Z = [\diag(u_1)Y, \cdots, \diag(u_r)Y] \in \R^{U \times Lr}$ for some matrix $Y \in \R^{U \times L}$ and $r$ vectors $u_1, \cdots, u_r \in \R^U$.
    \item \bf Output \rm: The inverse $(H + Z Z^{\top})^{-1} - H^{-1}$ under the assumption that $(H + Z Z^{\top})^{-1}$ exists. 
\end{itemize}
Naively computing this inverse using Woodbury's identity via 
\begin{equation} \label{eq: bottleneck_woodbury}
(H + ZZ^\top)^{-1} - H^{-1} = H^{-1}Z(I + Z^\top Z)^{-1}Z^\top H^{-1}
\end{equation}
takes $\Tmat(U,U,Lr)$ time. However, since the update matrix $Z$ only depends on $Ur + UL$ variables ($Y$ and $u_1, \cdots, u_r$), there is no known lower bound that prohibits a better running time of $\Tmat(U,U,L) + \Tmat(U,U,r)$.

Computing $H^{-1}Z$ in \eqref{eq: bottleneck_woodbury} reduces to the following fundamental problem: Given matrices $A, C \in \R^{n \times n}$ and vectors $b_1, \ldots b_n \in \R^n$ the task is to compute $D_k = A \diag(b_k)C$ for all $k \in [n]$. Let $d_{ijk} := (D_k)_{i,j}$.  Then using $b_{k\ell} := (b_k)_\ell$ we can write 
\begin{equation*}
    d_{ijk} = \sum_{\ell = 1}^n a_{i\ell} b_{k\ell} c_{\ell j}.
\end{equation*}
For symmetry let us harmlessly permute the entries in $c$ to obtain
\begin{equation*}
    d_{ijk} = \sum_{\ell = 1}^n a_{i\ell} b_{k\ell} c_{j\ell}.
\end{equation*}
Note that this is exactly the term for matrix multiplication extended to three matrices. For a matrix $F = A B^\top$ we would have $f_{ij} = \sum_{\ell} a_{i\ell}b_{j\ell}$.
From the paragraphs above it is clear that we can compute all $d_{ijk}$ in time $\Tmat(n,n,n^2)$. But can this be done faster? 
Interestingly, the problem of finding a 4-clique in a graph can be reduced to this problem, further proving its significance. Note that the currently best known running time for this problem matches $\Tmat(n,n,n^2)$ for $n$-vertex graphs \cite{eisenbrand2004complexity}.

\paragraph{Acknowledgments}
The second author would like to thank Vissarion Fisikopoulos and Elias Tsigaridas for introducing him from a practical perspective to Sum-of-Square optimization under the interpolant basis. 

\bibliography{bib}

\bibliographystyle{alpha}
\end{document}